\documentclass[a4paper,dvipsnames]{siamart1116}

\usepackage{lipsum}
\usepackage{amsfonts}
\usepackage{amsmath,amssymb}
\usepackage{graphicx}
\usepackage{epstopdf}
\usepackage{algorithmic}
\usepackage[caption=false]{subfig}

\ifpdf
  \DeclareGraphicsExtensions{.eps,.pdf,.png,.jpg}
\else
  \DeclareGraphicsExtensions{.eps}
\fi

\numberwithin{theorem}{section}

\newcommand{\TheTitle}{Globally convergent Jacobi-type algorithms \\ for simultaneous orthogonal\\ symmetric tensor diagonalization}
\newcommand{\TheShortTitle}{Globally convergent Jacobi-type algorithms}

\newcommand{\TheAuthors}{J. Li, K. Usevich, and P. Comon}

\headers{\TheShortTitle}{\TheAuthors}

\title{{\TheTitle}\thanks{Submitted to the editors DATE.
\funding{This work was supported by the ERC project ``DECODA'' no.320594, in the frame of the European program FP7/2007-2013.
The first author was partially supported by the National Natural Science Foundation of China (No.11601371).}}}

\author{
  Jianze Li\thanks{{School of Mathematics, Tianjin University, Tianjin 300072, China}
    (\email{lijianze@tju.edu.cn}).}
  \and
  Konstantin Usevich\thanks{GIPSA-Lab, CNRS and Univ. Grenoble Alpes, France (\email{firstname.lastname@gipsa-lab.fr}).}
  \and
  Pierre Comon\footnotemark[3]
}

\usepackage{amsopn}

\ifpdf
\hypersetup{
  pdftitle={\TheTitle},
  pdfauthor={\TheAuthors}
}
\fi

\usepackage[mathscr]{euscript}
\usepackage{cleveref}

\newtheorem{remark}[theorem]{Remark}

\newtheorem{example}{Example}

\newcommand{\tens}[1]{\boldsymbol{\mathcal{#1}}}
\newcommand{\tenselem}[1]{\mathcal{#1}}

\newcommand{\matr}[1]{\boldsymbol{#1}}

\newcommand{\set}[1]{\mathscr{#1}}
\newcommand{\T}{{\sf T}}        % transposition
\makeatletter
\newcommand{\rank}[1]{\mathop{\operator@font rank}\{#1\}}
\newcommand{\colrank}[1]{\mathop{\operator@font colrank}\{#1\}}
\newcommand{\krank}[1]{\mathop{\operator@font krank}\{#1\}}
\newcommand{\trace}[1]{\mathop{\operator@font trace}\{#1\}}
\newcommand{\Diag}[1]{\mathop{\operator@font Diag}\{#1\}}    % a diagonal matrix
\newcommand{\diag}[1]{\mathop{\operator@font diag}\{#1\}}    % a vector
\newcommand{\Span}[1]{\mathop{\operator@font Span}\{#1\}}    % a space
\newcommand{\argmin}{\mathop{\operator@font argmin}}

\newcommand{\offdiag}[1]{\mathop{\operator@font offdiag}\{#1\}}    % a vector
\newcommand{\Proj}[2]{\mathop{\operator@font Proj_{#1}}{#2}}
\newcommand{\ProjGrad}[2]{\mathop{{\operator@font Proj} \nabla}#1(#2)}
\makeatother

\newcommand{\eqdef}{\stackrel{\sf def}{=}}
\newcommand{\RR}{\mathbb{R}}

\newcommand{\NN}{\mathbb{N}}

\newcommand{\ON}[1]{\set{O}_{#1}}
\newcommand{\SON}[1]{\set{SO}_{#1}}

\newcommand{\Gmat}[3]{\matr{G}^{(#1,#2,#3)}}

\newcommand{\contr}[1]{\mathop{\bullet_{#1}}}   % Contraction

\newcommand{\jl}{\color{black}}

\definecolor{darkgreen}{RGB}{5,180,5}
\newcommand{\ku}{\color{black}}

\newcommand{\fin}{\color{black}}

\begin{document}

\maketitle

% REQUIRED
\begin{abstract}
In this paper, we consider a family of Jacobi-type algorithms for simultaneous orthogonal diagonalization problem of symmetric tensors.
For the Jacobi-based algorithm of [SIAM J. Matrix Anal. Appl., 2(34):651--672, 2013],
we prove its global convergence for simultaneous orthogonal diagonalization of symmetric matrices and 3rd-order tensors.
We also propose a new Jacobi-based algorithm in the general setting
and prove its global convergence for \jl sufficiently smooth functions.\fin
\end{abstract}

% REQUIRED

% REQUIRED

%\paragraph{Notation}
% Let $\ON{n} \subset \RR^{n\times n}$ denote  the orthogonal group  and $\SON{n}$ denote the special orthogonal group (the group of rotations).

\section{Introduction}\label{sec:introduction}
% !TEX root = jacobi_simax.tex
Higher-order tensor \jl decompositions \fin have attracted a lot of attention in the last two decades because of the applications in various disciplines,
including signal processing, numerical linear algebra and data analysis
\cite{Cichocki15:review,Comon14:introduction,Kolda09:review}.
The most popular decompositions include the Canonical Polyadic and Tucker decompositions,
\jl where additional constraints are often imposed, such as 
symmetry \cite{Comon08:symmetric}, nonnegativity \cite{Lim09:nonnegtive}
or orthogonality \cite{Kolda01:Orthogonal}.

An important class of tensor approximation \jl problems is \fin approximation  with orthogonality constraints.
In particular,
the orthogonal symmetric tensor diagonalization problem for 3rd and 4th-order cumulant tensors is in the core of Independent Component Analysis \cite{Como92:elsevier,Como94:sp,Como94:ifac}, \jl and is a popular way to solve blind source separation problems \fin in signal processing \cite{Como10:book}.
\jl In the same context,  simultaneous orthogonal matrix diagonalization \cite{Cardoso93:JADE} is widely used; simultaneous orthogonal tensor diagonalization for  slices of  4th-order cumulants is also used to solve  source separation problems \cite{Lathauwer01:ICA}. \fin

\paragraph{\jl Main notations\fin} Let $\RR^{n_1 \times \cdots \times n_d}\eqdef\RR^{n_1}\otimes\cdots\otimes\RR^{n_d}$
denote the space  of {$d$th-order tensors}. \ku
In  the paper, the tensors are typeset with a bold calligraphic font (e.g, $\tens{T}$, $\tens{W}$), and matrices are in bold (e.g, $\matr{Q}$, $\matr{U}$); \ku the elements of tensors or matrices are typeset as $\tenselem{T}_{ijk}$ (or sometimes $\tenselem{T}_{i,j,k}$) or $Q_{ij}$ respectively.  \fin

\jl For a tensor $\tens{T} \in \RR^{n_1 \times \cdots \times n_d}$  and a matrix $\matr{M} \in \mathbb{R}^{m \times n_k}$, their {$k$-mode product} \cite[subsection 2.5]{Kolda09:review} is  the tensor $(\tens{T} {\contr{k}} \matr{M}) \in \mathbb{R}^{n_1 \times \cdots \times n_{k-1} \times {m} \times n_{k+1} \times \cdots \times n_d}$ defined as
\[
(\tens{T} \contr{k} \matr{M}) _{i_1,...,i_d} \eqdef \sum_{j=1}^{n_k} \tenselem{T}_{i_1,\ldots,i_{k-1},j,i_{k+1}\ldots,i_d}  {M}_{i_1,j} 
\]
A tensor $\tens{T} \in \RR^{n\times \cdots \times n}$ is called symmetric if its entries do not change under any permutation of indices; its diagonal is, by definition, the vector
\[
\diag{\tens{T}} \eqdef \begin{bmatrix}\tenselem{T}_{1\ldots1} & \cdots & \tenselem{T}_{n\ldots n} \end{bmatrix}^{\T}.
\]
\fin
We will always denote by $\|\cdot\|$ the Frobenius norm of a tensor or a matrix, or the Euclidean norm of a vector.
\jl Finally, \fin let $\ON{n} \subset \RR^{n\times n}$ denote the orthogonal group, \jl that is, the set of orthogonal matrices. \fin
Let $\SON{n}\subset \ON{n}$ denote the special orthogonal group,
\jl  the set of orthogonal matrices with determinant $1$.\fin

\paragraph{\jl Problem statement\fin} In this paper, we consider the following simultaneous orthogonal symmetric tensor diagonalization problem.
Let $\{\tens{A}^{(\ell)}:1\leq\ell\leq m\}\subset\RR^{n \times \cdots \times n}$ be a set of symmetric tensors.
We wish to maximize
\begin{equation}\label{eq:sym_tensor_diagonalization}
\matr{Q}_{*}=\arg\max_{\matr{Q} \in \SON{n}} \sum_{\ell=1}^{m} \|\diag{\tens{W}^{(\ell)}}\|^2,
\end{equation}
where $\tens{W}^{(\ell)}=\tens{A}^{(\ell)} \contr{1} \matr{Q}^{\T} \cdots \contr{d} \matr{Q}^{\T}$
for $1\leq\ell\leq m$.
Problem \cref{eq:sym_tensor_diagonalization}
has the following well-known problems as special cases:
\begin{itemize}
\item orthogonal tensor diagonalization problem if $m=1$ and $d>2$;
\item simultaneous orthogonal matrix diagonalization problem if $m>1$ and $d=2$.
\end{itemize}

Several algorithms have been proposed in the literature to solve special cases of problem \cref{eq:sym_tensor_diagonalization}.
The first were the Jacobi CoM (Contrast Maximization)
algorithm for orthogonal diagonalization of 3rd and 4th-order symmetric tensors
\cite{Como92:elsevier,Como94:sp,Como94:ifac} and the JADE (Joint Approximate Diagonalization of Eigenmatrices) algorithm for simultaneous orthogonal matrix diagonalization \cite{Cardoso93:JADE}.
An algorithm for simultaneous orthogonal 3rd-order tensor diagonalization was proposed in
\cite{Lathauwer96:simultaneous}.
These Jacobi-type algorithms have been very widely used in applications \cite{Como10:book},
and have the advantage that Jacobi rotations can be computed by rooting low-order polynomials.
Nevertheless, up to our knowledge,
the convergence of these methods was not proved,
although it was often observed in practice \cite{Como94:ifac,LDL09:handbook}.

\paragraph{\jl Contribution\fin} In this paper, we consider several Jacobi-type algorithms to solve problem
\cref{eq:sym_tensor_diagonalization}, and \jl study their global convergence properties. \fin
By {\it global convergence},
we mean that, for any starting point,
the \jl whole \fin sequence of iterations produced by the algorithm always converges \jl to a limit point\footnote{This was also called \emph{single-point convergence} in \cite{Usch15:pjo}.}. \fin
\jl Note that the global convergence does not guarantee convergence to a global maximum, since the cost function is multimodal.\fin

First, we consider the Jacobi-based algorithm 
\jl proposed in \cite{IshtAV13:simax} \fin for best low multilinear rank approximation of 3rd-order symmetric tensors.
\jl
The algorithm uses a gradient-based order of Jacobi rotations, hence we call \fin
this algorithm  {\it Jacobi-G} in our paper.
For the Jacobi-G algorithm the convergence of subsequences of iterations to stationary points was established in \cite{IshtAV13:simax}.
We prove that, for problem \cref{eq:sym_tensor_diagonalization},
{Jacobi-G algorithm converges globally and the limit point is always a stationary point in the cases $d=2,3$.}
The proof is based on a variant \L{}ojasiewicz theorem developed in \cite{SU15:pro} for analytic submanifolds of Euclidean space.

Second, we propose a  \jl new \fin Jacobi-based algorithm inspired by proximal methods \cite{Pari14:proximal},
which is called the {\it Jacobi-PC} algorithm in this paper.
We show that Jacobi-PC algorithm always converges globally to a stationary point in the general setting.
In particular,
for cases $d=3,4$ of problem \cref{eq:sym_tensor_diagonalization},
Jacobi-PC algorithm allows for a simple algebraic solution to find the optimal Jacobi rotation,
as in the Jacobi CoM algorithm \cite{Como92:elsevier,Como94:sp,Como94:ifac}.
In addition, this algorithm does not need the order of Jacobi rotations introduced in Jacobi-G algorithm.
Finally,
the global convergence of Jacobi-PC algorithm is proved for \jl sufficiently smooth  functions. \fin
Therefore,
these results may be applied to other tensor approximation problems
(\emph{e.g.,} \ku nonsymmetric orthogonal diagonalization \cite{MoraV08:simax},\fin or  Tucker approximation of
symmetric \cite{IshtAV13:simax} or antisymmetric \cite{Bego16:antisymmetric} tensors).

\paragraph{\jl Organization\fin} The paper is organized as follows.
In \crefrange{sec:prob_statement}{sec:tensor_properties},
we present some more or less known results or easy extensions.
In \crefrange{sec:jacobi_G_global}{sec:experiments},
we show our main results of this paper.
In \cref{sec:prob_statement},
we introduce the abstract optimization problem on $\SON{n}$
and the general Jacobi algorithm to solve this abstract problem.
In \cref{sec:jacobi_G},
we recall the Jacobi-G algorithm proposed in \cite{IshtAV13:simax}
and its local convergence properties in the general setting.
In \cref{sec:tensor_properties},
we list some properties that are specific to orthogonal tensor diagonalization problem.
In \cref{sec:jacobi_G_global},
we prove the global convergence of Jacobi-G algorithms
for simultaneous orthogonal diagonalization of matrices and 3rd-rder tensors.
In \cref{sec:new_algo},
we propose the Jacobi-PC algorithm and prove its global convergence.
We derive formulas for optimal Jacobi rotations in the case of 3rd and 4th-order tensors.
In \cref{sec:experiments}, we provide some numerical experiments.

% The outline is not required, but we show an example here.
%The paper is organized as follows. Our main results are in
%\cref{sec:main}, our new algorithm is in \cref{sec:alg}, experimental
%results are in \cref{sec:experiments}, and the conclusions follow in
%\cref{sec:conclusions}.

\section{Problem statement and Jacobi rotations}\label{sec:prob_statement}
% !TEX root = jacobi_simax.tex
\subsection{Abstract optimization problem}
Let $\textit{f}: \SON{n} \to \RR$ be a function.
The abstract optimization problem is to find $\matr{Q}^* \in \SON{n}$ such that
\begin{equation}\label{eq:min_SO_n}
\matr{Q}^* = \arg \max_{\matr{Q} \in \SON{n}} \textit{f}(\matr{Q}).
\end{equation}

\begin{example}\label{ex:ATD}
Problem \cref{eq:sym_tensor_diagonalization} in \cref{sec:introduction} is a special case of problem \cref{eq:min_SO_n}.
Because
$\|\tens{A}^{(\ell)}\|=\|\tens{W}^{(\ell)}\|$ for any $1\leq\ell\leq m$,
problem \cref{eq:sym_tensor_diagonalization} is equivalent to find
\begin{equation*}\label{eq:sym_tensor_diagonalization-2}
\matr{Q}_{*}=\arg\min_{\matr{Q} \in \SON{n}} \sum_{\ell=1}^{m} \|\offdiag{\tens{W}^{(\ell)}}\|^2,
\end{equation*}
where $\offdiag{\tens{W}^{(\ell)}}$ is the vector of elements in $\tens{W}^{(\ell)}$ except diagonal elements.
\end{example}

\begin{example}\label{ex:Ishteva}
In \cite{IshtAV13:simax},
the best low multilinear rank approximation problem for 3rd order symmetric tensors
was formulated as a special case of problem \cref{eq:min_SO_n}.
In fact,
based on \cite[Theorem 4.1]{Lathauwer00:rank-1approximation},
the cost function has the form
\begin{equation}\label{eq:cost-function-IshtAV}
\textit{f}(\matr{Q}) = \|\tens{A} \contr{1} \matr{M}\matr{Q}^{\T}\contr{2} \matr{M}\matr{Q}^{\T}\contr{3} \matr{M}\matr{Q}^{\T}\|^2
\end{equation}
where $r<n$ and
$\matr{M}=\begin{bmatrix}
I_{r} &       0\\
  0 & 0
\end{bmatrix}
\in \RR^{n\times n}$.
\end{example}

\begin{example}\label{ex:Pierre}
Under some assumptions,
it was shown that the orthogonal tensor {diagonalization} problem
could be solved in the sense of maximization of the trace \jl of a tensor \fin\cite{Comon07:tensor}, \jl and is also a \fin special case of problem \cref{eq:min_SO_n}.
\end{example}

\subsection{Givens rotations and the general Jacobi algorithm}
Let $\theta\in\RR$ be an angle and $(i,j)$ be a pair of indices with $1 \le i < j \le n$.
We denote the Givens rotation matrix by
\[
\Gmat{i}{j}{\theta} =
\begin{bmatrix}
1 &       & & & &&\\
  &\ddots & & & & \matr{0} &\\
  & & \cos \theta & & -\sin\theta && \\
  & & & \ddots & &&\\
  & & \sin \theta & & \cos\theta && \\
  & \matr{0} & & & & \ddots &\\
 &       & & & &&1
\end{bmatrix},
\]
i.e., the matrix defined by
\[
(\Gmat{i}{j}{\theta})_{k,l} =
\begin{cases}
1, &  k = l, k\not\in \{i,j\}, \\
\cos{\theta}, &  k = l, k\in \{i,j\}, \\
\sin{\theta}, &  (k,l) = (j,i),\\
-\sin{\theta}, & (k,l) = (i,j),\\
0, &  \text{otherwise} \\
\end{cases}
\]
for $1 \le k,l \le n$.
We summarize the general Jacobi algorithm in \cref{alg:jacobi}.
\begin{algorithm}\caption{General Jacobi algorithm}\label{alg:jacobi}
{\bf Input:} A  function $f: \SON{n} \to \RR$, a starting point $\matr{Q}_{0}$.\newline
{\bf Output:} Sequence of iterations $\{\matr{Q}_{k}\}_{k\ge1}$.
\begin{itemize}
\item {\bf For} $k=1,2,\ldots$ [until a stopping criterion is satisfied] do
\item\quad Choose the pair $(i_k,j_k)$ according to a certain pair selection rule.
\item\quad Compute the angle $\theta^{*}_{k}$ that maximizes  the function $\ku h_k\fin(\theta)$ defined as
\begin{equation}\label{eq:h_k}
\ku h_k\fin(\theta)\eqdef{f}(\matr{Q}_{k-1} \Gmat{i_k}{j_k}{\theta}).
\end{equation}
\item \quad Set $\matr{U}_k \eqdef \Gmat{i_k}{j_k}{\theta^{*}_k}$, and update $\matr{Q}_k = \matr{Q}_{k-1} \matr{U}_k$.
\item {\bf End for}
\end{itemize}
\end{algorithm}

\jl
Note that in \cref{alg:jacobi} (and all the following algorithms, unless mentioned explicitly) we do not specify a stopping criterion, since our goal is to analyse the whole sequence of iterations $\{\matr{Q}_{k}\}_{k\ge1}$ produced by the algorithm. \ku
The stopping criterion applies only to  software implementation of algorithms.
\fin

\cref{alg:jacobi} \jl performs \fin the sequence of iterations $\matr{Q}_k$ by multiplicative updates:
\begin{equation}\label{eq:cumulative_product}
\matr{Q}_k = \matr{U}_1 \cdots \matr{U}_{k-1}\matr{U}_k,
\end{equation}
where each $\matr{U}_k$ is an elementary rotation.
The advantage of this Jacobi-type algorithm is that each update is a one-dimensional optimization problem, which can be solved efficiently.

\begin{remark}
Note that the choice of pair $(i_k,j_k)$ in every iteration is not specified in \cref{alg:jacobi}.
One of the most natural rules is in cyclic fashion as follows.
\begin{equation}\label{equation-Jacobi-C}
\begin{split}
&(1,2) \to (1,3) \to \cdots \to (1,n) \to \\
& (2,3) \to \cdots \to (2,n) \to \\
& \cdots  \to (n-1,n) \to \\
&(1,2) \to (1,3) \to \cdots.
\end{split}
\end{equation}
We call the Jacobi algorithm with this cyclic-by-row rule the \emph{Jacobi-C} algorithm.
\end{remark}

In Jacobi-C algorithm, the choice of pairs is periodic with the period $n(n-1)/2$.
Each set of $n(n-1)/2$ iterations is called a \emph{sweep}.
The pair selection rule \cref{equation-Jacobi-C} was used in the
Jacobi CoM algorithm \cite{Como92:elsevier,Como94:sp,Como94:ifac}
and JADE algorithm \cite{Cardoso93:JADE}.

\begin{remark}\label{rem:maximizer_choice}
If several equivalent maximizers are present in \cref{eq:h_k},
also in all the following algorithms, we choose the one with the angle of smaller magnitude.
\end{remark}

\section{Jacobi-G algorithm}\label{sec:jacobi_G}
% !TEX root = jacobi_simax.tex
In this section,
we recall the Jacobi-based algorithm in \cite{IshtAV13:simax} and its properties.
For simplicity,
this algorithm is called the {\it Jacobi-G} (gradient-based Jacobi-type) algorithm in our paper.

\subsection{Technical definitions}
Define the matrix
\[
\delta_{i,j} = \frac{d}{d \theta}\Gmat{i}{j}{\theta} \Big|_{\theta = 0} =
\begin{bmatrix}
0 &       & & & &&\\
  &\ddots & & & & \matr{0} &\\
  & & 0 & & -1 && \\
  & & & \ddots & &&\\
  & & 1 & & 0 && \\
  & \matr{0} & & & & \ddots &\\
 &       & & & &&0
\end{bmatrix}
\]
and introduce the notation
\[
d_{i,j}(\matr{Q}) \eqdef \matr{Q} \delta_{i,j}
\]
for $\matr{Q}\in\ON{n}$.
Next, for a differentiable function $f: \ON{n} \to \RR$,
we define the projected gradient \cite[Lemmma 5.1]{IshtAV13:simax}
as
\begin{equation}\label{eq:ProjGrad}
\ProjGrad{f}{\matr{Q}} \eqdef \matr{Q} \matr{\Lambda}(\matr{Q}),
\end{equation}
where
\begin{equation}\label{eq:Lambda}
\matr{\Lambda}(\matr{Q})  \eqdef \frac{\matr{Q}^{\T}\nabla\textit{f}(\matr{Q}) - (\nabla\textit{f}(\matr{Q}))^{\T}\matr{Q}}{2}
\end{equation}
and $\nabla\textit{f}(\matr{Q})$ is the Euclidean gradient of $f$ as a function of the matrix argument.
The projected gradient $\ProjGrad{f}{\matr{Q}}$ is exactly the Riemannian gradient
if $\ON{n}$ is viewed as an embedded submanifold of $\RR^{n \times n}$ \cite{Absil08:Optimization}.

\subsection{Jacobi-G algorithm}
In \cite{IshtAV13:simax}, \ku a modification of \cref{alg:jacobi}
was proposed, that choose  a pair $(i,j)$ at each iteration 
 that satisfies
\begin{equation}\label{eq:pair_selection_gradient}
|\langle \ProjGrad{f}{\matr{Q}}, d_{i,j}(\matr{Q})\rangle|  \ge \varepsilon \|\ProjGrad{f}{\matr{Q}}\|,
\end{equation}
where  $\varepsilon$ is a small positive constant.
\fin

\begin{algorithm}\caption{Jacobi-G algorithm}
\label{alg:jacobi-G}
{\bf Input:} A  function $f: \SON{n} \to \RR$,    a small positive $0<\varepsilon\leq\frac{2}{n}$, a starting point $\matr{Q}_{0}$.\newline
{\bf Output:} Sequence of iterations $\{\matr{Q}_{k}\}_{k\ge1}$.
\begin{itemize}
\item {\bf For} $k=1,2,\ldots$ [until a stopping criterion is satisfied] do
\item\quad Choose the pair $\ku (i,j) = \fin(i_k,j_k)$ satisfying inequality \eqref{eq:pair_selection_gradient} at $\ku\matr{Q} = \fin \matr{Q}_{k-1}$.
\item\quad Compute the angle $\theta^{*}_{k}$ that maximizes the function $\ku h_k\fin(\theta)$ defined in \eqref{eq:h_k}.
\item\quad  Set $\matr{U}_k \eqdef \Gmat{i_k}{j_k}{\theta^{*}_k}$, and update $\matr{Q}_k = \matr{Q}_{k-1} \matr{U}_k$.
\item {\bf End for}
\end{itemize}
\end{algorithm}

\ku
Note that, as shown in {\cite[Proof of Lemma 5.3]{IshtAV13:simax}} the left hand side of \cref{eq:pair_selection_gradient} can be also written in terms of the function $h_k$:
\begin{equation}\label{eq:gradient_projection_via_h_k}
|h'_k(0)| = |\langle \ProjGrad{f}{\matr{Q}_{k-1}}, d_{i_k,j_k}(\matr{Q}_{k-1})\rangle|.
\end{equation}
\fin

The following lemma (an easy generalisation of \cite[Lemma 5.2]{IshtAV13:simax}) shows that
it is always possible to choose such a pair  \ku $(i_k,j_k)$. \fin

\begin{lemma}\label{lemma-G-order-existence}
For any differentiable function $f: \SON{n} \to \RR$,
$\matr{Q}\in\SON{n}$ and $0 < \varepsilon \le 2/n$,
it is always possible to find $(i,j)$, with $i< j$, such that \eqref{eq:pair_selection_gradient} holds.
\end{lemma}
\begin{proof}
First of all, thanks to the representation \eqref{eq:ProjGrad}, rotation invariance of the Euclidean norm and  $\matr{\Lambda}(\matr{Q})$ being skew-symmetric, we have that the condition \eqref{eq:pair_selection_gradient} is equivalent to
\begin{equation}\label{eq:pair_selection_gradient_Lambda}
2 |\Lambda_{i,j}(\matr{Q})|  \ge \varepsilon \|\matr{\Lambda}(\matr{Q})\|.
\end{equation}
Now we choose the pair $(i,j)$ that maximizes the left-hand side of \eqref{eq:pair_selection_gradient_Lambda}
\[
(i,j) = \arg \max_{1\leq k, l \leq n} |\ku {\Lambda}_{k,l} \fin(\matr{Q})|,
\]
Since $\matr{\Lambda}(\matr{Q})$ is skew-symmetric, we can choose such a pair with $i < j$.
Finally, by the well known inequality between matrix norms,
\[
2 |\ku\Lambda_{i,j}\fin(\matr{Q})| =2\max_{1\leq k, l \leq n}\left|\ku {\Lambda}_{k,l} \fin(\matr{Q})\right|
\ge \frac{2}{n}\|\matr{\Lambda}(\matr{Q})\| \ge \varepsilon \|\matr{\Lambda}(\matr{Q})\|.
\]
\end{proof}
\begin{remark}
\ku
\cref{lemma-G-order-existence} was proved in  \cite[Lemma 5.2]{IshtAV13:simax} only for the cost function \cref{eq:cost-function-IshtAV},
\ku although it is valid for any differentiable function.
Hence, \fin the convergence results of \cite{IshtAV13:simax} are valid for a wide class of functions (see the following theorem).
\end{remark}

\begin{theorem}[{\cite[Theorem 5.4]{IshtAV13:simax}} combined with \cref{lemma-G-order-existence}]\label{theorem-ishteva-stationary-point}
Let $f$ be a $C^{\infty}$ function. Then every accumulation point\footnote{i.e., the limit of every convergent subsequence.}  $\matr{Q}_*$ of the sequence $\{\matr{Q}_{k}\}_{k\ge1}$ produced by \cref{alg:jacobi-G} is a stationary point of the function $f$ (\emph{i.e}, $\ProjGrad{f}{\matr{Q}_*} = 0$).
\end{theorem}

\subsection{Variants of the Jacobi-G algorithm}
Note that the description of \cref{alg:jacobi-G} does not say precisely how to select the pairs $(i_k,j_k)$
that satisfy \cref{eq:pair_selection_gradient}.
The first option is suggested by the proof in \cref{lemma-G-order-existence}: set $\varepsilon = \frac{2}{n}$
and select the maximal element in \cref{eq:Lambda}.
We summarize this option in \cref{alg:jacobi-G-max}.
\begin{algorithm}\caption{Jacobi-G-max algorithm}
\label{alg:jacobi-G-max}
{\bf Input:} A  function $f: \SON{n} \to \RR$, a starting point $\matr{Q}_{0}$.\newline
{\bf Output:} Sequence of iterations $\{\matr{Q}_{k}\}_{k\ge1}$.
\begin{itemize}

\item {\bf For} $k=1,2,\ldots$ until a stopping criterion is satisfied do
\item\quad Choose the pair $(i_k,j_k) = (i,j)$ that maximizes $|\ku {\Lambda}_{i,j} \fin(\matr{Q}_{k-1})|$ in \cref{eq:Lambda}.
\item\quad Compute the angle $\theta^{*}_{k}$ that maximizes  the function $\ku h_k\fin(\theta)$ defined in \eqref{eq:h_k}.
\item\quad  Set $\matr{U}_k \eqdef \Gmat{i_k}{j_k}{\theta^{*}_k}$, and update $\matr{Q}_k = \matr{Q}_{k-1} \matr{U}_k$.
\item {\bf End for}
\end{itemize}
\end{algorithm}

However, choosing the maximal element in \cref{eq:Lambda} requires a search over all the elements,
which may take additional time.
In $\cite{IshtAV13:simax}$,
it was suggested to take $\varepsilon \ll 1$.
\jl If we choose $(i_k,j_k)$ in the cyclic order and $\varepsilon$  is very small, then it is natural to expect that the inequality \cref{eq:pair_selection_gradient} will be often satisfied, and thus the behavior of the algorithm will be very close to the behavior of the Jacobi-C algorithm. \fin

\jl
To make this idea more rigorous, we construct a modification of Jacobi-C algorithm that skips the rotations if the magnitude of the directional derivative is below a given threshold. The modification is described in \cref{alg:jacobi-G-threshold}. \fin

\begin{algorithm}\caption{Jacobi-C-threshold  algorithm}
\label{alg:jacobi-G-threshold}
{\bf Input:} A  function $f: \SON{n} \to \RR$, a threshold $\delta>0$, a starting point $\matr{Q}_{0}$.
\newline
{\bf Output:} A point $\widehat{\matr{Q}}$, such that $\|\ProjGrad{f}{\widehat{\matr{Q}}} \|\le \delta$.
\begin{itemize}
\item Set  $k=1$.
\item {\bf Do } (cyclic sweep)
\item\quad {\bf For} $i=1,\ldots,n$, $j = i+1,\ldots, n$
\item\quad \quad {\bf If} $|\ku \Lambda_{i,j} \fin(\matr{Q}_{k-1})| > \frac{\delta}{n}$
\item\quad \quad\quad Set $(i_k,j_k) = (i,j)$.
\item\quad\quad\quad Compute the angle $\theta^{*}_{k}$ that maximizes   $h_{\ku k\fin}(\theta)$ defined in \eqref{eq:h_k}.
\item\quad\quad\quad  Set $\matr{U}_k \eqdef \Gmat{i_k}{j_k}{\theta^{*}_k}$, and update $\matr{Q}_k = \matr{Q}_{k-1} \matr{U}_k$.
\item\quad \quad \quad Set $k \leftarrow k+1$.
\item\quad \quad {\bf End If}
\item\quad {\bf End For}
\item {\bf While} there was progress in the last sweep.
\item Return $\widehat{\matr{Q}} = \matr{Q}_{k-1}$.
\end{itemize}
\end{algorithm}

\medskip

\jl
Note that the output of \cref{alg:jacobi-G-threshold} is different from other algorithms in terms of the output, because the rule of skipping the rotations also yields a well-defined stopping criterion.
Also, note that, compared to Jacobi-G, the inequality \cref{eq:pair_selection_gradient} is not needed. 
But, \cref{alg:jacobi-G-threshold} can be viewed as a special case of \cref{alg:jacobi-G}, as shown by the following remark.
\fin

\begin{remark}\label{prop:jacobi_threhold}
There exists $\varepsilon = \varepsilon(\delta)$ such that the iterates $\matr{Q}_k$ produced by \cref{alg:jacobi-G-threshold} are the first elements of the sequence produced by a Jacobi-G algorithm (\cref{alg:jacobi-G}). 
Indeed, consider $c= \max_{\matr{Q} \in \SON{n}} \|\ProjGrad{f}{{\matr{Q}}}\|$, which is finite if $f$ is smooth (since $\SON{n}$ is compact).
Hence  $\varepsilon = \frac{\delta}{nc}$ is the required value of $\varepsilon$.

\jl
Also note that, \fin \cref{alg:jacobi-G-threshold} always terminates if the Jacobi-G algorithm converges to a stationary point.
\end{remark}

\section{\ku Derivatives in the orthogonal  tensor diagonalization problems \fin}\label{sec:tensor_properties}
% !TEX root = jacobi_simax.tex
\subsection{Projected gradient: \ku the case of a single tensor ($m=1$) \fin}
\label{sec:proj_grad}
In this subsection,
we derive a concrete form of projected gradient \cref{eq:ProjGrad},
which will be used in \cref{sec:jacobi_G_global}.
Let $\tens{A}\in\RR^{n \times \cdots \times n}$ be a $d$th-order symmetric tensor and
$\matr{Q}\in\ON{n}$ be an orthogonal matrix.
Let
$\tens{W} =  \tens{A} \contr{1} \matr{Q}^{\T} \cdots \contr{d} \matr{Q}^{\T}$
and
\begin{equation}\label{eq:cost-function-diagonalization}
\textit{f}(\matr{Q}) = \sum\limits_{j=1}^n \tenselem{W}^2_{jj\ldots j}
=\sum\limits_{j=1}^n(\sum\limits_{p_1,p_2,\ldots,p_d}\tenselem{A}_{p_1,p_2,\ldots,p_d}Q_{p_1,j}Q_{p_2,j}\ldots Q_{p_d,j})^2.
\end{equation}
\ku First, we \fin calculate the Euclidean gradient of $\textit{f}$ at $\matr{Q}$.
Let us fix $i$ and $j$.
Then
\begin{align*}
\frac{\partial \textit{f}}{\partial Q_{i,j}}
&=2\tenselem{W}_{jj\ldots j}\Big[dQ_{i,j}^{d-1}\tenselem{A}_{i,i,\ldots,i}
+(d-1)\binom{d}{d-1}Q_{i,j}^{d-2}(\sum_{k_1\neq i}Q_{k_1,j}\tenselem{A}_{i,\ldots,i,k_1}) + \ldots\\
&\phantom{=2\tenselem{W}_{jj\ldots j}[[}+q\binom{d}{q}Q_{i,j}^{q-1}(\sum\limits_{k_1,\ldots,k_{d-q}\neq i}Q_{k_1,j}\cdots Q_{k_{d-q},j}
\tenselem{A}_{i,\ldots,i,k_1,\ldots,k_q})+\ldots\\
&\phantom{=2\tenselem{W}_{jj\ldots j}[[}+d\sum\limits_{k_1,\ldots,k_{d-1}\neq i}
Q_{k_1,j}\cdots Q_{k_{d-1},j}\tenselem{A}_{i,k_1,\ldots,k_{d-1}})\Big]\\
%%%%%%
&= 2d\tenselem{W}_{jj\ldots j}\sum\limits_{k_1,\ldots,k_{d-1}=1}^{n}
Q_{k_1,j}\cdots Q_{k_{d-1},j}\tenselem{A}_{i,k_1,\ldots,k_{d-1}} =
2d\tenselem{W}_{jj\ldots j} \tenselem{V}_{ij\ldots j},
\end{align*}
where $\tens{V} = \tens{A} \contr{2} \matr{Q}^{\T} \cdots \contr{d} \matr{Q}^{\T}$.
Noting that $\tens{V} = \tens{W}\contr{1} \matr{Q}$, we get that
\begin{align*}
\nabla \textit{f}(\matr{Q})&=  2d
  \begin{bmatrix}
    \tenselem{V}_{11\ldots 1} & \tenselem{V}_{12\ldots 2} & \cdots & \tenselem{V}_{1n\ldots n}\\
    \tenselem{V}_{21\ldots 1} & \tenselem{V}_{22\ldots 2} & \cdots & \tenselem{V}_{2n\ldots n}\\
   \vdots&\cdots&\cdots&\vdots\\ \\
   \tenselem{V}_{n1\ldots 1} & \tenselem{V}_{n2\ldots 2} & \cdots & \tenselem{V}_{nn\ldots n}
  \end{bmatrix}
    \begin{bmatrix}
    \tenselem{W}_{1\ldots 1} & 0 & \cdots & 0\\
    0 & \ddots & \ddots & \vdots\\
    \vdots&\ddots&\ddots&0\\
    0 & \cdots &0 & \tenselem{W}_{n\cdots n}
  \end{bmatrix}\\
&=  2d
\matr{Q}
  \begin{bmatrix}
    \tenselem{W}_{11\ldots 1} & \tenselem{W}_{12\ldots 2} & \ldots & \tenselem{W}_{1n\ldots n}\\
    \tenselem{W}_{21\ldots 1} & \tenselem{W}_{22\ldots 2} & \ldots & \tenselem{W}_{2n\ldots n}\\
   \ldots&\ldots&\ldots&\ldots\\ \\
   \tenselem{W}_{n1\ldots 1} & \tenselem{W}_{n2\ldots 2} & \ldots & \tenselem{W}_{nn\ldots n}
  \end{bmatrix}
   \begin{bmatrix}
    \tenselem{W}_{1\ldots 1} & 0 & \cdots & 0\\
    0 & \ddots & \ddots & \vdots\\
    \vdots&\ddots&\ddots&0\\
    0 & \cdots &0 & \tenselem{W}_{n\cdots n}
  \end{bmatrix}.
\end{align*}
After projecting $\nabla\textit{f}(\matr{Q})$ onto the tangent space at
$\matr{Q}$ to the  manifold $\ON{n}$,
we get
\[
\ProjGrad{f}{\matr{Q}}=\matr{Q} \matr{\Lambda}(\matr{Q}),
\]
where $\matr{\Lambda}(\matr{Q})$ is the matrix with
\begin{equation}\label{eq:Projgrad-general}
\ku\Lambda_{k,l}\fin(\matr{Q})=
\begin{cases}
0, &  k = l;\\
d(\tenselem{W}_{kl\ldots l}\tenselem{W}_{l\ldots l}\text{-}\tenselem{W}_{k\ldots k}\tenselem{W}_{k\ldots kl}), &  k < l;\\
-\ku\Lambda_{k,l}\fin(\tens{W}), &  k > l.
\end{cases}
\end{equation}
for any $1\leq k,l\leq n$.
This is a special case of $\cref{eq:Lambda}$ for function \cref{eq:cost-function-diagonalization}.

\begin{remark}\label{projected-gradient-3}
Let $\tens{A}$ be a 3rd order symmetric tensor.
Then $\ProjGrad{f}{\matr{Q}}=$
\begin{align*}
3\matr{Q}\left[\begin{smallmatrix}
    0 & \tenselem{W}_{122}\tenselem{W}_{222} -\tenselem{W}_{111}\tenselem{W}_{112}   &\tenselem{W}_{133}\tenselem{W}_{333} - \tenselem{W}_{111}\tenselem{W}_{113}  \\
    \tenselem{W}_{111}\tenselem{W}_{112}  - \tenselem{W}_{122}\tenselem{W}_{222}  & 0 & \tenselem{W}_{233}\tenselem{W}_{333} - \tenselem{W}_{222}\tenselem{W}_{223}  \\
   \tenselem{W}_{111}\tenselem{W}_{113} - \tenselem{W}_{133}\tenselem{W}_{333}& \tenselem{W}_{222}\tenselem{W}_{223}-\tenselem{W}_{233}\tenselem{W}_{333}  & 0
   \end{smallmatrix}\right].
\end{align*}
\end{remark}

\subsection{The cost function at each iteration}\label{sec:cost_function_iteration}
Now consider a single iteration in \cref{alg:jacobi} \ku for the cost function \eqref{eq:cost-function-diagonalization} \fin . For simplicity\footnote{other cases follow by substitution of indices}, we consider only the case $(i_k,j_k) = (1,2)$.
We also denote $\matr{G}(\theta)\eqdef \Gmat{1}{2}{\theta}$ without loss of generality.
Then $\matr{Q}_k = \matr{Q}_{k-1} \matr{G}(\theta^{*}_{k})$.

In this subsection, we take $\matr{Q} =  \matr{Q}_{k-1}$,  \ku  so that
\[
\tens{W}=\tens{A} \contr{1}\matr{Q}^{\T}\contr{2}\matr{Q}^{\T} \cdots\contr{d}\matr{Q}^{\T},
\]
is the rotated tensor before the $k$th iteration.
We also use notation $\tens{T}$ for the candidate tensors after $k$-th iteration, i.e.
\[
\tens{T} = \tens{T}(\theta) =\tens{W}\contr{1}\matr{G}^{\T}(\theta)\contr{2}\matr{G}^{\T}(\theta)\cdots \contr{d}\matr{G}^{\T}(\theta),
\]
so that $h_k(\theta) = \|\diag{\tens{T}(\theta)} \|^2$.
Note that we omit $k$ in the notation for $\tens{W}$ and $\tens{T}$, but \fin
this will not lead to confusion within one iteration.

%\subsection{Periodicity of the cost function}\label{subsection-Periodicity}
Although $\ku h_k \fin (\theta)$  is defined on the whole real line,
it is periodic.
Apart from the obvious period $2\pi$, it has a smaller period $\pi/2$.
Indeed,
\begin{align*}
\matr{G}\left(\theta+\pi/2\right)=
\begin{bmatrix}
-\sin(\theta) & -\cos(\theta) & 0 & \cdots & 0\\
\cos(\theta) & -\sin(\theta) & 0 & \cdots & 0\\
0 & 0 & 1 & \cdots & 0\\
\vdots & \vdots & \vdots & \ddots & \vdots\\
0 & 0 & 0 & \cdots & 1
\end{bmatrix}
=
 \matr{G}(\theta) \begin{bmatrix}
0 & -1 & 0 & \cdots & 0\\
1 & 0 & 0 & \cdots & 0\\
0 & 0 & 1 & \cdots & 0\\
\vdots & \vdots & \vdots & \ddots & \vdots\\
0 & 0 & 0 & \cdots & 1
\end{bmatrix}.
\end{align*}
Hence, the tensor $\tens{T}(\theta+\pi/2)$ differs from $\tens{T}(\theta)$ by permutations of the indices $1$ and $2$ and change of signs, which does not change the cost function.
Hence, due to \cref{rem:maximizer_choice}, we are, in fact, maximizing $\ku h_k\fin(\theta)$ on the interval  $[-\pi/4, \pi/4]$ .
\ku In fact, in all the algorithms we choose \fin $\ku \theta^{*}_k\fin\in [-\pi/4, \pi/4]$ with $\ku h_k\fin(\ku \theta^{*}_k\fin)=\max\limits_{\theta\in\RR} \ku h_k\fin(\theta)$.

It is often convenient to rewrite the function $\ku h_k\fin(\theta)$ in \cref{eq:h-tensor-problem} in a polynomial form.
Consider the change of variables $\theta = \arctan(x)$. Then minimization of $\ku h_k\fin(\theta)$
on $[-\pi/4, \pi/4]$ is equivalent to minimization of
\[
\ku \tau_k\fin(x) \eqdef \ku h_k\fin(\arctan(x))
\]
on $[-1,1]$. After the change of variables, since $x = \tan \theta$, we have
\begin{equation}\label{eq:subst_tan_G}
 \matr{G}(\theta) =
 \begin{bmatrix} \matr{P}(\theta) & \\
 & \matr{I}_{n-2} \end{bmatrix},
 \quad \mbox{where }\matr{P}(\theta) =  \frac{1}{\sqrt{x^2+1}} \begin{bmatrix}
1 & -x \\
x & 1
\end{bmatrix}.
\end{equation}
Hence, the cost function is a rational function in $x$
\begin{equation}\label{eq:tau_rational}
\tau_k(x) = \frac{\rho(x)}{(1+x^2)^d}, 
\end{equation}
where $\rho(x)$ is a polynomial of degree $2d$ \ku (note that we dropped the index $k$ for simplicity). 

In {\cite[\ku eqn. \fin (22)-(23)]{Como94:ifac}} it was shown that finding critical points of  $\tau_k(x)$ can be reduced to finding roots of a quadratic polynomial if $d=3$ or a quartic polynomial if $d=4$.
We do not provide these expressions in this section, but their generalizations can be found in \cref{sec:algebraic_proximal}.
\fin

\subsection{Derivatives for matrices and third-order tensors}\label{sec:deriv_2_3}
\ku
We recall the expressions for derivatives of the cost function \eqref{eq:cost-function-diagonalization} in the case $d \in \{2,3\}$  from \cite{Como94:ifac},
but give proofs for completeness.
 \fin
\begin{lemma}[{\cite[\ku eqn. \fin (22)-(23)]{Como94:ifac}}]\label{lem:deriv_order_3}
In the notations of \cref{sec:cost_function_iteration}, the derivatives of $h_{\ku k \fin}$  \ku have the following form. \fin
\begin{itemize}
\item \ku For $d=3$:  \fin
%
%\begin{align*}
%\ku {h}_{k}^{'} \fin(\theta)&= 6(\tens{T}_{111}\tens{T}_{112}-\tens{T}_{122}\tens{T}_{222}),\\
%\ku {h}_{k}^{''} \fin(\theta) &=
%-6(\tens{T}_{111}^2+\tens{T}_{222}^2-3\tens{T}_{112}^2-3\tens{T}_{122}^2
%-2\tens{T}_{111}\tens{T}_{122}-2\tens{T}_{112}\tens{T}_{222}).
%\end{align*} \ku
%In particular, \fin
\begin{align*}
& h^{'}_{\ku k\fin} (\theta) = 6(\tenselem{T}_{111}\tenselem{T}_{112}-\tenselem{T}_{122}\tenselem{T}_{222}),\\
&h^{''}_{\ku k\fin}(\theta) = -6(\tenselem{T}_{111}^2+\tenselem{T}_{222}^2-3\tenselem{T}_{112}^2-3\tenselem{T}_{122}^2
-2\tenselem{T}_{111}\tenselem{T}_{122}-2\tenselem{T}_{112}\tenselem{T}_{222}).
\end{align*}
\item \ku For $d=2$:
\begin{align*}
& h^{'}_{\ku k\fin} (\theta) = 4(\tenselem{T}_{11}\tenselem{T}_{12}-\tenselem{T}_{12}\tenselem{T}_{22}),\\
&h^{''}_{\ku k\fin}(\theta) = -4(\tenselem{T}_{11}^2+\tenselem{T}_{22}^2-2\tenselem{T}_{11}\tenselem{T}_{22}-4\tenselem{T}_{12}^2). \fin
\end{align*}
\end{itemize}
\ku
The expressions for $h^{'}_{k} (0)$ and $h^{''}_{k} (0)$  can be found by substituting $\tens{T}$ to $\tens{W}$ in the above expressions.
\fin
\end{lemma}
\begin{proof}
\ku From \eqref{eq:Projgrad-general} \fin %eqn. % \cref{projected-gradient-3},
we have
\begin{align*}
&h^{'}_{\ku k\fin}(\theta)=\langle\ProjGrad{\textit{f}}{\matr{Q}_{k-1}\matr{G}(\theta)}, \matr{Q}_{k-1}\matr{G}^{'}(\theta)\rangle
=\langle \matr{Q}_{k-1}\matr{G}(\theta)\matr{\Lambda}(\matr{Q}_{k-1}\matr{G}(\theta)), \matr{Q}_{k-1}\matr{G}^{'}(\theta)\rangle\\
&=\langle\matr{\Lambda}(\matr{Q}_{k-1}\matr{G}(\theta)), {\matr{G}(\theta)}^{T}\matr{G}^{'}(\theta)\rangle 
= -2\ku{\Lambda}_{1,2}\fin(\matr{Q}_{k-1}\matr{G}(\theta)) \\
& = \begin{cases} 
6(\tenselem{T}_{111}\tenselem{T}_{112}-\tenselem{T}_{122}\tenselem{T}_{222}), & d = 3, \\
4(\tenselem{T}_{11}\tenselem{T}_{12}-\tenselem{T}_{12}\tenselem{T}_{22}), & d=2,
\end{cases}
\end{align*}
where $\Lambda$ is  defined in \cref{eq:Projgrad-general}.

\ku After straightforward differentiation, we have that for $d=3$
\[
\frac{d\tenselem{T}_{111}}{d\theta}=3\tenselem{T}_{112}, 
\frac{d\tenselem{T}_{112}}{d\theta}=2\tenselem{T}_{122}-\tenselem{T}_{111},
\frac{d\tenselem{T}_{122}}{d\theta}=\tenselem{T}_{222}-2\tenselem{T}_{111},
\frac{d\tenselem{T}_{222}}{d\theta}=-3\tenselem{T}_{122},
\]
and for $d=2$
\[
\frac{d\tenselem{T}_{11}}{d\theta} = 2\tenselem{T}_{12},\quad
\frac{d\tenselem{T}_{12}}{d\theta} = (\tenselem{T}_{22} - \tenselem{T}_{11}),\quad\frac{d \tenselem{T}_{22} }{d\theta} = -2\tenselem{T}_{12}.
\]
The equation for $h^{''}_{\ku k\fin}(\theta)$ follows by substitution.
\end{proof}

\subsection{The general case ($m > 1$)}\label{sec:deriv_general}
\ku In the general case,  the cost function in problem \cref{eq:sym_tensor_diagonalization} is \fin
\begin{equation}\label{eq:sim_diag_cost}
{f}(\matr{Q})=\sum\limits_{\ell=1}^{m}{f}^{\ku(\ell)\fin}\fin(\matr{Q}),
\end{equation}
where $\ku{f}^{\ku(\ell)\fin}(\matr{Q})=\|\diag{\tens{W}^{(\ell)}}\|^2$
for any $1\leq\ell\leq m$.

\ku By linearity, the cost function $h_k(\theta)$ in \cref{eq:h_k} in every iteration can be conveniently written as \fin
\begin{equation}\label{eq:h-tensor-problem}
h_{\ku k\fin}(\theta) = \sum\limits_{\ell=1}^{m} \ku h^{(\ell)}_{k}\fin(\theta),
\end{equation}
\ku where \fin
\[
\ku h_{k}^{(\ell)}\fin(\theta) \eqdef {f}^{\ku(\ell)\fin}(\matr{Q}_{k-1}\Gmat{i_k}{j_k}{\theta}) =  \|\diag{\tens{T}^{(\ell)}}(\theta)\|^2,
\]
\ku
and $\tens{T}^{(\ell)}(\theta)$ is defined in the same way as $\tens{T}$ in \cref{sec:cost_function_iteration}. Therefore, the derivatives of $\ku h_{k}^{(\ell)}\fin$ can be obtained in the same way as in \cref{sec:deriv_2_3}

Finally, as in  \cref{sec:cost_function_iteration}, we can use a change of variables 
\begin{equation}\label{eq:tangent_change}
\tau_k(x) = h_k(\arctan(x)),
\end{equation}
 which leads to 
\[
\tau_k(x) = \sum\limits_{\ell=1}^{m} \tau^{(\ell)}_k(x)
\]
where, for $1\leq\ell\leq m$,
\[
\tau^{(\ell)}_k(x) = h^{(\ell)}_k(\arctan(x)) = \frac{\rho^{(\ell)}(x)}{(1+x^2)^d}.
\]

\fin

\section{Global convergence of Jacobi-G algorithm for symmetric low-order tensors}\label{sec:jacobi_G_global}
% !TEX root = jacobi_simax.tex

\subsection{\L{}ojasiewicz gradient inequality}
In this subsection, we recall some important results about the convergence of iterative algorithms.
The discrete-time analogue of classical \L{}ojasiewicz's theorem was proposed in \cite{AbsMA05:sjo},
and make it possible to prove the convergence of many algorithms
\cite{Usch15:pjo,Chu14:GlobalALS}.
In \cite{SU15:pro},
to prove the convergence of projected line-search methods on the real-algebraic variety
of real $m\times n$ matrices of rank at most $k$,
the optimization problem
\begin{equation*}
\min\limits_{x\in\mathcal{M}}\textit{f}(x)
\end{equation*}
on a closed subset $\mathcal{M}\subseteq\RR^n$ was considered.
Suppose that the tangent cone ${T}_{x} \mathcal{M}$ at $x\in\mathcal{M}$ is a linear space.
Let $\ProjGrad{f}{x}$ be the  projection of the Euclidean gradient $\nabla\textit{f}(x)$ on the tangent space at $x$.
{We first introduce the definition of an analytic submanifold in $\RR^n$.}

\begin{definition}[{\cite[Def. 2.7.1]{krantz2002primer}}]
\jl A set $\mathcal{M}\subseteq\RR^n$ is called an $m$-dimensional {\it real analytic submanifold}
if, for each $p\in\mathcal{M}$, there exists an open subset $\mathcal{U}\subseteq\RR^m$ and a real analytic
function $f:\mathcal{U}\rightarrow\RR^n$ which maps open subsets of $\mathcal{U}$ onto relatively open subsets
of $\mathcal{M}$ and which is such that
\[
p\in f(\mathcal{U})\ \ \text{and}\ \ \text{rank}\,\matr{J}_f(u)=m,\ \ \forall u\in \mathcal{U},
\]
where $\matr{J}_f(u)$ is the Jacobian matrix of $f$ at $u$. \fin
\end{definition}

The following results were proved in \cite{SU15:pro}.

\begin{lemma}\label{lemma-SU15}
Let $\mathcal{M}\subseteq\RR^n$ be an analytic submanifold.
Then any point $x\in \mathcal{M}$ satisfies a \L{}ojasiewicz inequality for
$\ProjGrad{f}{x}$,
that is,
there exist $\delta>0$, $\sigma>0$ and $\zeta\in (0,1/2]$
such that for all $y\in\mathcal{M}$ with $\|y-x\|<\delta$,
it holds that
$$|\textit{f}(x)-\textit{f}(y)|^{1-\zeta}\leq \sigma\|\ProjGrad{f}{x}\|.$$
\end{lemma}

\begin{theorem}[{\cite[Theorem 2.3]{SU15:pro}}]\label{theorem-SU15}
Let $\mathcal{M}\subseteq\RR^n$ be an analytic submanifold
and $\{x_k:k\in\NN\}\subset\mathcal{M}$ be a sequence.
Suppose that $f$ is real analytic and, for large enough $k$,\\
(i) there exists $\sigma>0$ such that
$$|\textit{f}(x_{k+1})-\textit{f}(x_k)|\geq \sigma\|\ProjGrad{f}{x_k}\|\|x_{k+1}-x_{k}\|;$$
(ii) $\ProjGrad{f}{x_k}=0$ implies that $x_{k+1}=x_{k}$.\\
Then any accumulation point of $\{x_k:k\in\NN\}\subseteq\mathcal{M}$ is the only limit point.
\end{theorem}

Now we apply \cref{theorem-SU15} to the compact orthogonal group
$\ON{n}\subset\RR^{n\times n}$ and get \cref{theorem-convegence-general},
which \jl will allow us \fin to prove the global convergence of Jacobi-G algorithm in \cref{subsection-Jacobi-G-3}.

\begin{corollary}\label{theorem-convegence-general}
Let f be real analytic in \cref{alg:jacobi}.
Suppose that, for large enough k,\\
(i) there exists $\sigma>0$ such that
\begin{equation}\label{condition-coro-KL}
|\textit{f}(\matr{Q}_{k})-\textit{f}(\matr{Q}_{k-1})|\geq \sigma\|\ProjGrad{f}{\matr{Q}_{k-1}}\| \|\matr{Q}_{k}-\matr{Q}_{k-1}\|;
\end{equation}
(ii) $\ProjGrad{f}{\matr{Q}_{k-1}}=0$ implies that $\matr{Q}_{k}=\matr{Q}_{k-1}$.\\
Then the iterations $\{\matr{Q}_k:k\in\NN\}$ converge to a point $\matr{Q}_*\in\ON{n}$.
\end{corollary}

\begin{remark}
Under the same assumptions as in \cref{theorem-convegence-general},
\cref{theorem-ishteva-stationary-point} tells us that \cref{alg:jacobi-G} converges to a stationary point.
\end{remark}

\subsection{Global convergence of Jacobi-G algorithm for \ku matrices and \fin 3rd-order tensors}\label{subsection-Jacobi-G-3}
In this section, we consider the case $d \in \{2,3\}$, that is, one of the following options:
\begin{itemize}
\item \ku for a set of 3rd-order symmetric tensors \fin
 $\{\tens{A}^{(\ell)}:1\leq\ell\leq m\}\subseteq\RR^{n\times n\times n}$, the cost function  is
\begin{equation}\label{order-3-cost-function}
f(\matr{Q}) =
\sum\limits_{\ell=1}^{m}\|\diag{\tens{A}^{(\ell)} \contr{1}\matr{Q}^{\T}\contr{2}\matr{Q}^{\T}\contr{3}\matr{Q}^{\T}}\|^2;
\end{equation}
\item for a set $\{\matr{A}^{(\ell)}:1\leq\ell\leq m\}\subseteq\RR^{n\times n}$ of symmetric matrices, the cost function  is
\begin{equation}\label{order-2-cost-function}
f(\matr{Q})= \sum_{\ell=1}^{m}
\|\diag{\matr{Q}^{\T} \matr{A}^{(\ell)} \matr{Q}}\|^2.
\end{equation}

\end{itemize}

\begin{theorem}\label{theorem-convergence-order-3}
For the cost function  \cref{order-3-cost-function} or \cref{order-2-cost-function},
\cref{alg:jacobi-G} converges to a stationary point of $f$ in $\ON{n}$, for any starting point $\matr{Q}_{0}$.
\end{theorem}
\begin{remark}
In the case $m=1$ \ku and $d=3$, \fin
this is the Jacobi-G algorithm for orthogonal diagonalization of 3rd order symmetric tensors.
\cref{theorem-convergence-order-3} shows the global convergence of this algorithm.
\end{remark}

Before giving the proof of \cref{theorem-convergence-order-3}, we formulate several lemmas.

\begin{lemma}
In the case \ku $d\in\{2,3\}$,  \fin for the cost function $\tau_{\ku k\fin}(x)$ \ku defined in \cref{eq:tangent_change}, the following identities hold true
\begin{equation}\label{equation-Jacobi-G-1}
 \tau_{\ku k\fin} \fin(x)- \tau_{\ku k\fin} (0)=\frac{1}{(1+x^2)^2}
({h}^{'}_{\ku k\fin} (0)(x-x^3)+\frac{1}{2}{h}^{''}_{\ku k\fin} (0)x^2),
\end{equation}
\begin{equation}\label{equation-Jacobi-G-2}
\tau^{'}_{\ku k\fin} (x) = \frac{1}{(1+x^2)^3}(\textit{h}^{'}_{\ku k\fin} (0)(1-6x^2+x^4)+\textit{h}^{''}_{\ku k\fin} (0)(x-x^3)).
\end{equation}
\end{lemma}
\begin{proof}
\ku
First, by linearity of the expressions \cref{equation-Jacobi-G-1} and \cref{equation-Jacobi-G-2}, and from \cref{sec:deriv_general} we can prove the identities only for the case of a single tensor or matrix (i.e., the cost function \eqref{eq:cost-function-diagonalization}).
Second, the equality  \cref{equation-Jacobi-G-2} follows from \cref{equation-Jacobi-G-1} by straightforward differentiation and the fact that \fin
\[
\tau^{'}_{\ku k\fin}(x) = (\tau_{\ku k\fin}(x) -\tau_{\ku k\fin}(0))'.
\]
\ku 
Hence, we are left to prove \cref{equation-Jacobi-G-1} for the cost function \eqref{eq:cost-function-diagonalization}.

Recall the notation of \cref{sec:cost_function_iteration}, and consider the case $d=3$. By substitution \cref{eq:subst_tan_G}, and due to the fact that the rotation affects only first two elements on the diagonal, we get that \fin
\begin{align*}
&\tau_{\ku k\fin}(x)-\tau_{\ku k\fin}(0)=
\textit{h}_{\ku k\fin}(\theta)-\textit{h}_{\ku k\fin}(0)
= \tenselem{T}_{111}^2+\tenselem{T}_{222}^2-\tenselem{W}_{111}^2-\tenselem{W}_{222}^2\\
&=\frac{1}{(1+x^2)^3}[6(\tenselem{W}_{111}\tenselem{W}_{112}-\tenselem{W}_{122}\tenselem{W}_{222})(x-x^5)\\
&-3(\tenselem{W}_{111}^2+\tenselem{W}_{222}^2-3\tenselem{W}_{112}^2-3\tenselem{W}_{122}^2
-2\tenselem{W}_{111}\tenselem{W}_{122}-2\tenselem{W}_{112}\tenselem{W}_{222})(x^2+x^4)]\\
&=\frac{1}{(1+x^2)^2}[h^{'}_{\ku k\fin}(0)(x-x^3)+\frac{1}{2}h^{''}_{\ku k\fin}(0)x^2],
\end{align*}
\ku where the last equality follows from \cref{lem:deriv_order_3}. 

The case $d=2$ is analogous: we have \fin
\begin{align*}
\tau_{\ku k\fin}(x)&-\tau_{\ku k\fin}(0)=
\tenselem{T}_{11}^2+\tenselem{T}_{22}^2-\tenselem{W}_{11}^2-\tenselem{W}_{22}^2\\
&= \frac{2}{(1+x^2)^2}
[(2\tenselem{W}_{11}\tenselem{W}_{12} - 2\tenselem{W}_{12}\tenselem{W}_{22})(x-x^3)\\
&-(\tenselem{W}_{11}^2 + \tenselem{W}_{22}^2 - 2\tenselem{W}_{11}\tenselem{W}_{22} - 4\tenselem{W}_{12}^2)x^2]\\
&=\frac{1}{(1+x^2)^2}
[\textit{h}^{'}_{\ku k\fin}(0)(x-x^3)+\frac{1}{2}\textit{h}^{''}_{\ku k\fin}(0)x^2],
\end{align*}
\ku
where the last equality follows again from  \cref{lem:deriv_order_3}.
\fin
\end{proof}

\begin{lemma}\label{lemma-G-inequality}
In each iteration of \cref{alg:jacobi-G} for the cost function \cref{order-3-cost-function} or \cref{order-2-cost-function}, the following inequality holds true
\begin{align*}
|f(\matr{Q}_{k})-f(\matr{Q}_{k-1})|\geq \frac{\sqrt{2}\varepsilon}{4}\|\ProjGrad{f}{\matr{Q}_{k-1}}\|\cdot\|\matr{Q}_{k}-\matr{Q}_{k-1}\|
\end{align*}
for any $k\in\NN$.
\end{lemma}
\begin{proof}
\ku At each iteration, from \cref{eq:pair_selection_gradient}  and \cref{eq:gradient_projection_via_h_k} \fin we have
\begin{equation}\label{eq:lemma-G-inequality-1}
| h^{'}_{\ku k\fin}(0)| 
\geq \varepsilon\|\ProjGrad{\textit{f}}{\matr{Q}_{k-1}}\|.
\end{equation}
Next, for an optimal angle $\theta_* \ku = \theta^*_k \fin$, we have
\begin{equation}\label{eq:lemma-G-inequality-2}
\|\matr{Q}_{k}-\matr{Q}_{k-1}\|=\|\ku \Gmat{i_k}{j_k}{\theta_*} \fin - \matr{I}_n\| = 2\sqrt{2}\left|\sin\left(\frac{\theta_{*}}{2}\right)\right |\leq\sqrt{2}|\theta_{*}|.
\end{equation}
Note that the tangent $x_* = \tan(\theta_*)$ should satisfy equation $\tau^{'}_{\ku k\fin}(x_{*})=0$.

If $x_{*}= 0$ or $\pm1$, then $h^{'}_{\ku k\fin} (0)=0$ from \cref{equation-Jacobi-G-2}, \ku hence $\|\ProjGrad{\textit{f}}{\matr{Q}_{k-1}}\| = 0$ from \cref{eq:lemma-G-inequality-1}  and \fin the result is obvious.
Consider the case $0<|x_{*}|<1$. Then from \cref{equation-Jacobi-G-2} we get
\begin{equation}\label{eq:two-derivative}
h^{''}_{\ku k\fin}(0)=\frac{h^{'}_{\ku k\fin}(0)(x_{*}^4-6x_{*}^2+1)}{x_{*}(x_{*}^2-1)},
\end{equation}
and thus
\begin{equation}\label{eq:lemma-G-inequality-3}
\tau_{\ku k\fin}(x_{*})-\tau_{\ku k\fin}(0) = \frac{x_{*}h^{'}_{\ku k\fin}(0)}{2(1-x_{*}^2)}
\end{equation}
by substituting \cref{eq:two-derivative} into (\ref{equation-Jacobi-G-1}).
Finally, by combining \crefrange{eq:lemma-G-inequality-1}{eq:lemma-G-inequality-3}, we get
\begin{align*}
&|f(\matr{Q}_{k})-f(\matr{Q}_{k-1})|=|h_{\ku k\fin}(\theta_{*})-h_{\ku k\fin}(0)|= \left|\frac{x_{*}h^{'}_{\ku k\fin}(0)}{2(1-x_{*}^2)}\right|\\
&\geq\frac{1}{2}|\tan(\theta_{*})h^{'}_{\ku k\fin}(0)|
\geq \frac{\sqrt{2}\varepsilon}{4}\|\ProjGrad{\textit{f}}{\matr{Q}_{k-1}}\|\cdot\|\matr{Q}_{k}-\matr{Q}_{k-1}\|
\end{align*}
\end{proof}

\begin{proof}[Proof of \cref{theorem-convergence-order-3}]
\cref{lemma-G-inequality} guarantees that condition \cref{condition-coro-KL} in \cref{theorem-convegence-general} holds true.
Since the cost function is analytic, by \cref{theorem-convegence-general}, the sequence $\matr{Q}_k$ converges to a stationary point $\matr{Q}_*$.
Finally, by \cref{theorem-ishteva-stationary-point}, $\matr{Q}_*$ is a stationary point of $f$ in \cref{order-3-cost-function}.
\end{proof}

\section{Jacobi-PC algorithm and \ku its global \fin convergence}\label{sec:new_algo}
% !TEX root = jacobi_simax.tex
The Jacobi-G algorithm has several disadvantages:
the convergence for $4$th-order tensors is currently unknown,
and the parameter $\varepsilon$ needs to be chosen in a proper way.
In this section, we propose an new Jacobi-based algorithm,
which is inspired by proximal algorithms in convex \ku \cite{Pari14:proximal} and nonconvex \cite{Bolte14:Proximal} optimization . \fin

\subsection{Jacobi-PC algorithm and its global convergence}
Suppose that we are given a \jl twice continuously differentiable \fin  function $f: \SON{n} \to \RR$, such that
\begin{equation}\label{eq:pi2_periodicity}
f(\matr{Q} \matr{G}^{(i,j,\theta)}) = f(\matr{Q} \matr{G}^{(i,j,\theta + \pi/2)})
\end{equation}
for any $\matr{Q} \in \SON{n}$ and $1\leq i<j\leq n$ (\emph{i.e.}, it is $\pi/2$-periodic along any geodesic).
Then we propose the {\it Jacobi-PC} algorithm (Jacobi-C algorithm with a proximal term) in \cref{alg:jacobi-C-P}.

\begin{algorithm}\caption{Jacobi-PC algorithm}\label{alg:jacobi-C-P}
{\bf Input:} A smooth function $f: \SON{n} \to \RR$,   and a  positive $\delta_0$, a starting value $\matr{Q}_{0}$.\newline
{\bf Output:} Sequence of iterations $\{\matr{Q}_{k}\}_{k\ge1}$.
\begin{itemize}
\item {\bf For} $k=1,2,\ldots$ until a stopping criterion is satisfied do
\item\quad Choose the pair $(i_k,j_k)$ according to the pair selection rule \eqref{equation-Jacobi-C}.
\item\quad Compute the angle $\theta^{*}_{k}$ that maximizes  the function
\[
\tilde h_{\ku k\fin}(\theta) =  \textit{f}(\matr{Q}_{k-1}\Gmat{i_k}{j_k}{\theta})-\delta_0 \gamma(\theta),
\]
where
\[
\gamma(\theta) = 2\sin^2(\theta)\cos^2(\theta).
\]
\item\quad  Set $\matr{U}_k \eqdef \Gmat{i_k}{j_k}{\theta^{*}_k}$, and update $\matr{Q}_k = \matr{Q}_{k-1} \matr{U}_k$.
\item {\bf End for}
\end{itemize}
\end{algorithm}
\begin{remark}
The periodicity condition \cref{eq:pi2_periodicity} is not necessary for the global convergence of the algorithm in \cref{th:convergence-jacobi-C},
but we add it due to its presence in the orthogonal tensor diagonalization problem.
If the condition \cref{eq:pi2_periodicity} does not hold, another proximal term $\gamma(\theta)$ may be needed.
Finally, other pair selection {\color{blue}rules} than \cref{equation-Jacobi-C} can be used.
\end{remark}

\begin{theorem}\label{th:convergence-jacobi-C}
The sequence produced by
\cref{alg:jacobi-C-P} converges to a stationary point $\matr{Q}_*\in\ON{n}$
for any starting point $\matr{Q}_{0}$.
\end{theorem}
\begin{proof}
We first prove the convergence.
Since
\[
f(\matr{Q}_k)  - f(\matr{Q}_{k-1}) -\delta_0 \gamma(\theta^*_k) =\tilde h_{\ku k\fin}(\theta^*_k) - \tilde h_{\ku k\fin}(0) \ge 0,
\]
we get that
\begin{equation}\label{eq:proximal_inequality}
f(\matr{Q}_k) - f(\matr{Q}_{k-1}) \ge \delta_0 \gamma(\theta^*_k) \ge 0.
\end{equation}
Note that $f(\matr{Q}_{k})$ is bounded since $\ON{n}$ is compact.
Then $f(\matr{Q}_{k})\rightarrow c<+\infty$ and thus
$$f(\matr{Q}_k) - f(\matr{Q}_{k-1}) \rightarrow  0.$$
By \cref{eq:proximal_inequality},
we have that
$\gamma(\theta^*_k) \rightarrow 0.$
Note that $\gamma(\theta) \ge 8|\theta|^2/\pi^2$ for $\theta \in [-\pi/4, \pi/4]$.
Then $\theta^*_k \to 0$ and thus there exists $\matr{Q}_*\in\ON{n}$ such that $\matr{Q}_k \to \matr{Q}_*$.

Next we prove that $\tilde{h}^{'}_{\ku k\fin}(0)\rightarrow 0$,
that is $\ku\Lambda_{i_k,j_k}\fin(\matr{Q}_{k-1})\rightarrow 0$.
Define
\[
\bar{h}(\theta\ku, \matr{Q}\fin) =  \textit{f}(\matr{Q}\Gmat{i}{j}{\theta})-\delta_0 \gamma(\theta)
\]
for $\theta\in\RR$, $\matr{Q}\in \ON{n}$ and $1\leq i<j\leq n$.
Let
\begin{align*}
M_1\eqdef\max\limits_{\substack{\matr{Q}\in \ON{n},\theta\in\RR,\\1\leq i<j\leq n}} \left| \ku\frac{\partial^2\bar{h}}{(\partial \theta)^2} \fin (\theta\ku, \matr{Q}\fin) \right|.
\end{align*}
Then $M_1<+\infty$ since \ku $f$ is C$^2$ \fin smooth, $\bar{h}$ is periodic \ku with respect to $\theta$ and $\ON{n}$ is compact.  \fin
Therefore, we have that
\[
|\tilde{h}^{'}_{\ku k\fin}(0)|=|\tilde{h}^{'}_{\ku k\fin}(\theta^{*}_{k})-\tilde{h}^{'}_{\ku k\fin}(0)|\leq|\theta^{*}_{k}|M_1
\]
for any $\matr{Q}_{k-1}\in \ON{n}$, \ku $\theta^{*}_{k} \in \RR$ \fin  and $1 \leq i_k < j_k \leq n$,
and thus $\tilde{h}^{'}_{\ku k\fin}(0)\rightarrow 0$.

Finally we prove that $\matr{Q}_*\in\ON{n}$ is a stationary point of $\textit{f}$,
that is $\matr{\Lambda}(\matr{Q}_{k-1})\rightarrow 0$.
We have proved that $\ku\Lambda_{i_k,j_k}\fin(\matr{Q}_{k-1})\rightarrow 0$ {in the above part}.
Now we prove other entries of $\matr{\Lambda}(\matr{Q}_{k-1})$ also converge to 0.
For simplicity,
take $(i_k,j_k)=(1,2)$ and $(i_{k+1},j_{k+1})=(1,3)$ for instance.
It is enough to prove that $\ku\Lambda_{1,2}\fin(\matr{Q}_{k})\rightarrow 0$.
In fact,
if we define
\begin{align*}
\phi: \RR\rightarrow\RR,\ \theta\mapsto \ku\Lambda_{1,2}\fin(\matr{Q}_{k-1}\matr{G}(\theta)),
\end{align*}
then $\ku\Lambda_{1,2}\fin(\matr{Q}_{k}) = \phi(\theta^{*}_{k})$ and $\ku\Lambda_{1,2}\fin(\matr{Q}_{k-1}) = \phi(0)$.
Define
\[
\bar{\phi}(\theta) =  \ku\Lambda_{1,2}\fin(\matr{Q}\Gmat{i}{j}{\theta})
\]
for $\theta\in\RR$, $\matr{Q}\in \ON{n}$ and $1\leq i<j\leq n$.
Let
\begin{align*}
M_2\eqdef\max\limits_{\substack{\matr{Q}\in \ON{n},\theta\in\RR,\\1\leq i<j\leq n}}|\bar{\phi}^{'}(\theta)|.
\end{align*}
Then $M_2<+\infty$ since $\bar{\phi}$ is smooth and periodic.
Therefore
\begin{align*}
|\ku\Lambda_{1,2}\fin(\matr{Q}_{k})-\ku\Lambda_{1,2}\fin(\matr{Q}_{k-1})|=|\phi(\theta^{*}_{k})-\phi(0)|\leq|\theta^{*}_{k}|M_2,
\end{align*}
and thus $\ku\Lambda_{1,2}\fin(\matr{Q}_{k})\rightarrow 0$.
\end{proof}

\subsection{Elementary rotations for orthogonal tensor diagonalization}\label{sec:algebraic_proximal} 
\ku
The cost function \cref{eq:sim_diag_cost}  in  simultaneous orthogonal tensor diagonalizatiom has the property \cref{eq:pi2_periodicity}, hence \fin
the Jacobi-PC  algorithm is guaranteed to converge. \ku
Moreover, it allows for finding \fin the update using an algebraic algorithm  \ku in the cases $d=3,4$. \fin

%Let $\tens{A}\in\RR^{n \times \cdots \times n}$ be a $d$th-order symmetric tensor.
%Problem \cref{eq:sym_tensor_diagonalization} in this case is the orthogonal tensor diagonalization problem,
%that is to find
%\begin{equation}\label{equation-orthogonal-tensor-diagonalizaiton}
%\matr{Q}_{*}=\arg\max_{\matr{Q} \in \SON{n}}
%\|\diag{\tens{A}\contr{1}\matr{Q}\cdots\contr{d}\matr{Q}}\|^2.
%\end{equation}
\ku Let us show how to find \fin $\theta^{*}_{k}$ in every iteration of \cref{alg:jacobi-C-P}.
Let $\widetilde{\tau}_{\ku k\fin}(x) = \widetilde{h}_{\ku k\fin}(\arctan x )$ be as in \cref{sec:cost_function_iteration}. \ku
Then we obtain that \fin
\[
\widetilde{\tau}_{\ku k\fin}(x) = \frac{\rho(x)}{(1+x^2)^d}-2\delta_{0}\frac{x^2}{(1+x^2)^2},
\]
where $\rho(x)$ is the polynomial defined in  \cref{eq:tau_rational}.
Then $\widetilde{\tau}^{'}_{\ku k\fin}(x)=0$ is equivalent to
$\omega(x)=0$,
where
$$\omega(x)=(1+x^2)^{d+1}\widetilde{\tau}^{'}_{\ku k\fin}(x)=\rho^{'}(x)(1+x^2)-2dx\rho(x)-4\delta_{0}x(1-x^2)(1+x^2)^{d-2}$$
is a polynomial of degree $2d$.

Note that from \ku \cref{sec:cost_function_iteration} and $\frac{\pi}{2}$-periodicity of $\gamma$, we have that
$\widetilde{h}_{\ku k\fin}(\theta)=\widetilde{h}_{\ku k\fin}(\theta+\pi/2)$ \fin for any $\theta\in\RR$, \ku hence \fin $\widetilde{\tau}_{\ku k\fin}(x)=\widetilde{\tau}_{\ku k\fin}(-1/x)$.
Now we represent the algebraic solutions of $\omega(x)$ by this property,
that is,
$\omega(x)=0$ has the same solutions as $\omega(-1/x)=0$ except the possible roots at the origin.
Let $\xi=x-1/x$.
Then
\begin{align*}
\omega(x)&=x^{2(d-d_1-d_2)}(1+x^2)^{d_2}\prod\limits_{{\ku j\fin}=1}^{d_1}(x-x_{\ku j\fin})(x+1/x_{\ku j\fin})\\
&=x^{2(d-d_1-d_2)}(1+x^2)^{d_2}\prod\limits_{{\ku j\fin}=1}^{d_1}(x^2-\xi_{\ku j\fin}x-1)=x^{2d-d_1-2d_2}(1+x^2)^{d_2}\prod\limits_{{\ku j\fin}=1}^{d_1}(\xi-\xi_{\ku j\fin})
\end{align*}
for some $0\leq d_1,d_2\leq d$.
Now we have that $\omega(x)=0$ if and only if
$$\Omega(\xi)=\prod\limits_{{\ku j\fin}=1}^{d_1}(\xi-\xi_{\ku j\fin})=0,$$
except the possible roots at the origin.
If the algebraic roots $\xi_{\ku j\fin}$ can be calculated
then the roots
$(x_{\ku j\fin},-1/x_{\ku j\fin})$ could be deduced by rooting the polynomials
$x^2-\xi_{\ku j\fin}x-1=0$.

\ku Now we restrict ourselves to the case of a single tensor (i.e. the cost function \cref{eq:cost-function-diagonalization}). \fin
In fact,
if $\tens{A}$ is of 3rd or 4th-order,
it \ku can \fin be shown that $\Omega(\xi)$ has algebraic solutions and thus
$(x_{\ku j\fin},-1/x_{\ku j\fin})$ \ku can \fin be determined. \ku
The following \cref{rem:coefficient-Jacobi-C} provides the specific form of $\Omega(\xi)$  in these cases, and is a direct generalisation of the results the ordinary Jacobi algorithm in \cite[Appendix]{Como94:ifac}. \fin

\begin{lemma}\label{rem:coefficient-Jacobi-C}
(i) Let $\tens{A}\in\RR^{n\times n\times n}$ be a 3rd order symmetric tensor
and
\begin{align*}
a &= 6(\tenselem{W}_{111}\tenselem{W}_{112}-\tenselem{W}_{122}\tenselem{W}_{222});\\
b &= 6(\tenselem{W}_{111}^2+\tenselem{W}_{222}^2-3\tenselem{W}_{112}^2-3\tenselem{W}_{122}^2
-2\tenselem{W}_{111}\tenselem{W}_{122}-2\tenselem{W}_{112}\tenselem{W}_{222})+4\delta_{0}.
\end{align*}
Then
\begin{align*}
\omega(x) &= a(1-5x^2-5x^4+x^6)+b(x^5-x) = x^2(1+x^2)[a\xi^2+b\xi-4a];\\
\Omega(\xi) &= a\xi^2+b\xi-4a.
\end{align*}\\
(ii) Let $\tens{A}\in\RR^{n\times n\times n\times n}$ be a 4th order symmetric tensor
and
\begin{align*}
a &= 8(\tenselem{W}_{1111}\tenselem{W}_{1112}-\tenselem{W}_{1222}\tenselem{W}_{2222});\\
b &=8(\tenselem{W}_{1111}^2-3\tenselem{W}_{1122}\tenselem{W}_{1111}-4\tenselem{W}_{1112}^2
-4\tenselem{W}_{1222}^2+\tenselem{W}_{2222}^2-3\tenselem{W}_{1122}\tenselem{W}_{2222})\!+\!4\delta_{0};\\
c &= 8(18\tenselem{W}_{1112}\tenselem{W}_{1122}-7\tenselem{W}_{1111}\tenselem{W}_{1112}+3\tenselem{W}_{1111}\tenselem{W}_{1222}\\
&-18\tenselem{W}_{1122}\tenselem{W}_{1222}-3\tenselem{W}_{1112}\tenselem{W}_{2222}+7\tenselem{W}_{1222}\tenselem{W}_{2222});\\
d &= 8(9\tenselem{W}_{1111}\tenselem{W}_{1122}-32\tenselem{W}_{1112}\tenselem{W}_{1222}-2\tenselem{W}_{1111}\tenselem{W}_{2222}\\
&+9\tenselem{W}_{1122}\tenselem{W}_{2222}+12\tenselem{W}_{1112}^2-36\tenselem{W}_{1122}^2+12\tenselem{W}_{1222}^2)+4\delta_{0};\\
e &= 80(6\tenselem{W}_{1122}\tenselem{W}_{1222}-\tenselem{W}_{1111}\tenselem{W}_{1222}-6\tenselem{W}_{1112}\tenselem{W}_{1122}+\tenselem{W}_{1112}\tenselem{W}_{2222}).
\end{align*}
Then
\begin{align*}
\omega(x) &= a(x^8+1)+b(x^7-x)+c(x^6+x^2)+d(x^5-x^3)+ex^4 \\
&= x^4[a(x^4+\frac{1}{x^4})+b(x^3-\frac{1}{x^3})+c(x^2+\frac{1}{x^2})+d(x-\frac{1}{x})+e]\\
&= x^4(a\xi^4+b\xi^3+(4a + c)\xi^2+(3b + d)\xi+2a+2c+e);\\
\Omega(\xi) &= a\xi^4+b\xi^3+(4a + c)\xi^2+(3b + d)\xi+2a+2c+e.
\end{align*}
\end{lemma}
\ku 
Note that if we set $\delta_0 = 0$, we obtain exactly the expressions from \cite[Appendix]{Como94:ifac}.
\fin

\begin{remark}
The expressions for  $\Omega(\xi)$  in the case of simultaneous orthogonal diagonalization problem can be also easily found in the same way as in \cref{rem:coefficient-Jacobi-C}, by exploiting the additivity of the corresponding expressions in \cref{sec:deriv_general}.
\end{remark}
\fin

\section{Numerical results}\label{sec:experiments}
% !TEX root = jacobi_simax.tex
In this section, we present numerical experiments in order to
compare the presented algorithms in the case of {\jl orthogonal diagonalization problems for  symmetric tensors. \fin}
The algorithms were implemented in MATLAB and the codes are available on request.

The setup of all the experiments is as follows:
\begin{itemize}
\item A diagonal tensor $\tens{D}$ is chosen. (For convenience, we choose the tensors such that $\|\tens{D}\| = 1$.)
\item A random rotation matrix $\matr{Q}$ is applied to obtain
\[
\tens{A}_0 = \tens{D} \contr{1} \matr{Q}^{\T}\cdots\contr{d}\matr{Q}^{\T}.
\]
\item The test tensor is constructed as $\tens{A} = \tens{A}_0 + \tens{E}$,
where $\tens{E}$ is the symmetrization of a tensor containing realization of i.i.d. Gaussian noise with variance $\sigma^2$.
\end{itemize}
To each test example we apply the following algorithms:
\begin{itemize}
\item \emph{Jacobi-C}: \cref{alg:jacobi} with the order of pairs \cref{equation-Jacobi-C}.
\item \emph{Jacobi-G-max}: \cref{alg:jacobi-G-max}.
\item \emph{Jacobi-G}: \cref{alg:jacobi-G} for various values of $\varepsilon$.
\item \emph{Jacobi-PC}: \cref{alg:jacobi-C-P} for various values of $\delta_0$ (shortened to ``Jacobi-P'' in the plots).
\end{itemize}
The stopping criterion is chosen to be the maximum number of iteration.

In each of the plots, we plot $\|\tens{A}\|^2 - f(\matr{Q}_k)$, which is exactly the squared norm of the off-diagonal elements.
In all the plots, the markers correspond to the places where the new sweep starts.

\subsection{Test 1: equal values on the diagonal}\label{sec:test-1}
In this subsection, we consider $10\times 10\times 10$ and $10\times 10 \times 10\times 10$ tensors where the diagonal values are given by
\[
\tenselem{D}_{i\ldots i} = \frac{1}{\sqrt{10}}.
\]
We plot the results in \cref{fig:test1_sig4,fig:test1_sig2}.
\begin{figure}[tbhp]
\centering
\subfloat[$3$-rd order]{\includegraphics[width=0.5\textwidth]{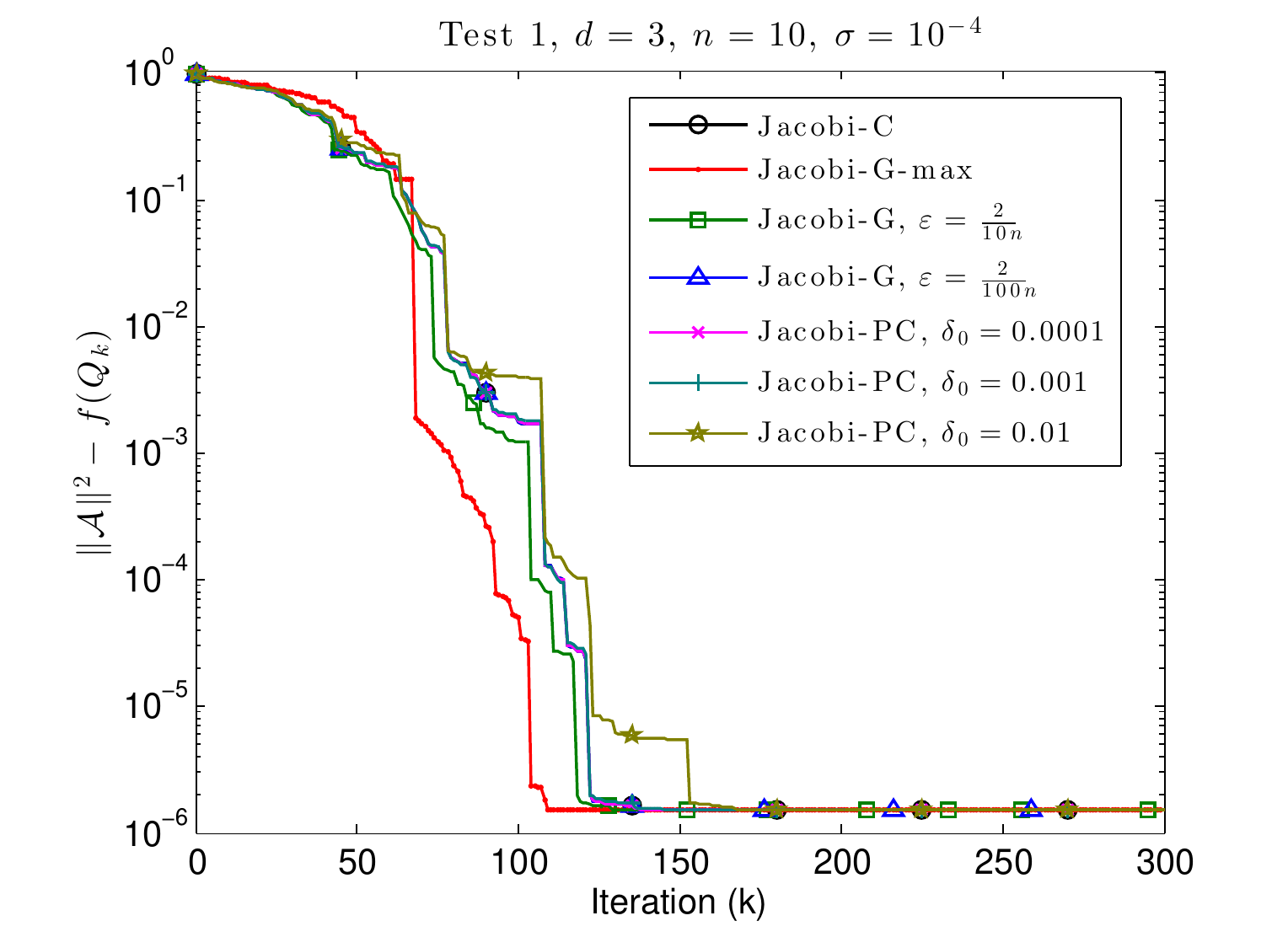}}\!\!\!
\subfloat[$4$-th order]{\includegraphics[width=0.5\textwidth]{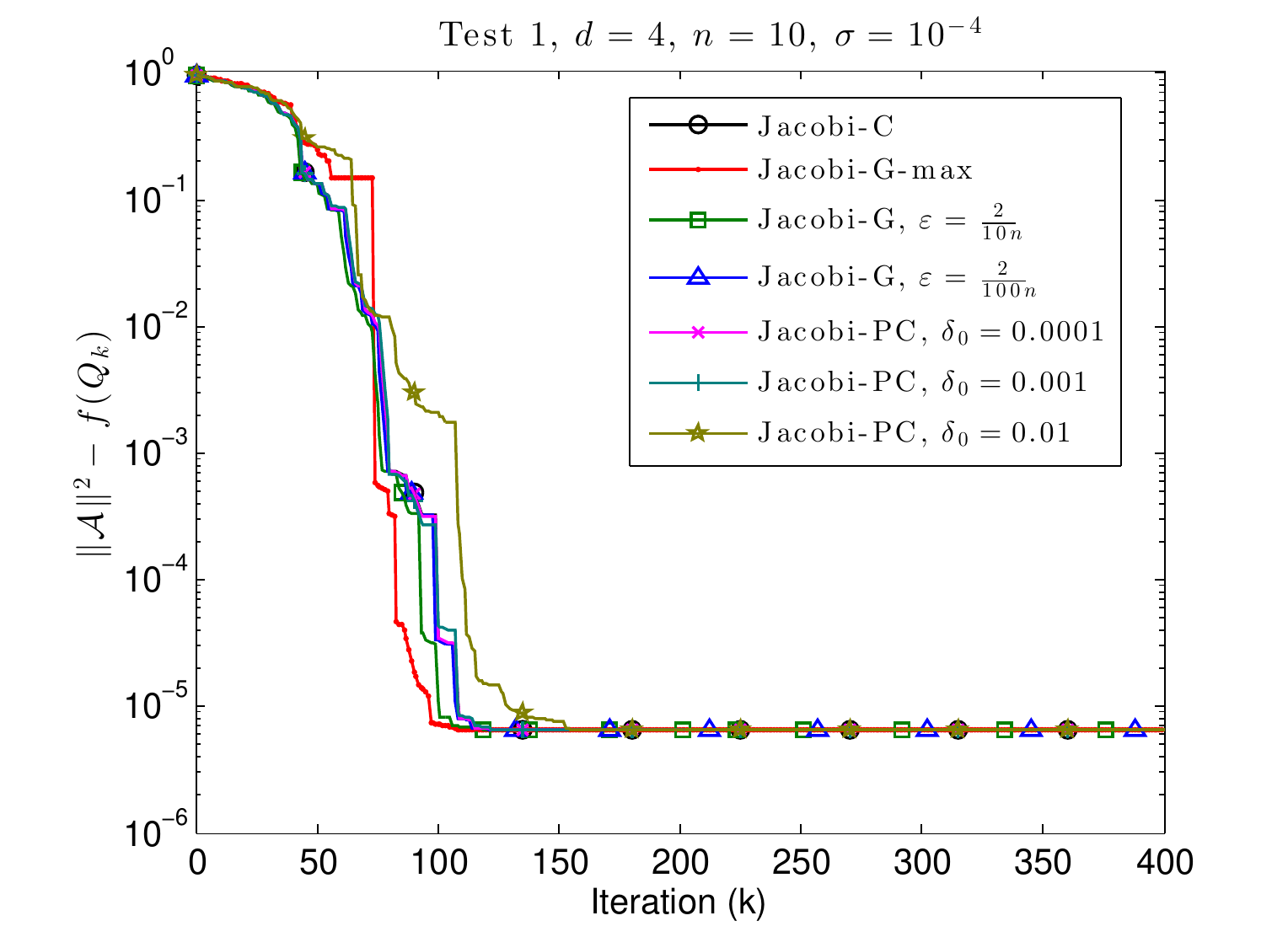}}
\caption{Equal values on the diagonal, small noise.}
\label{fig:test1_sig4}
\end{figure}
\begin{figure}[tbhp]
\centering
\subfloat[$3$-rd order]{\includegraphics[width=0.5\textwidth]{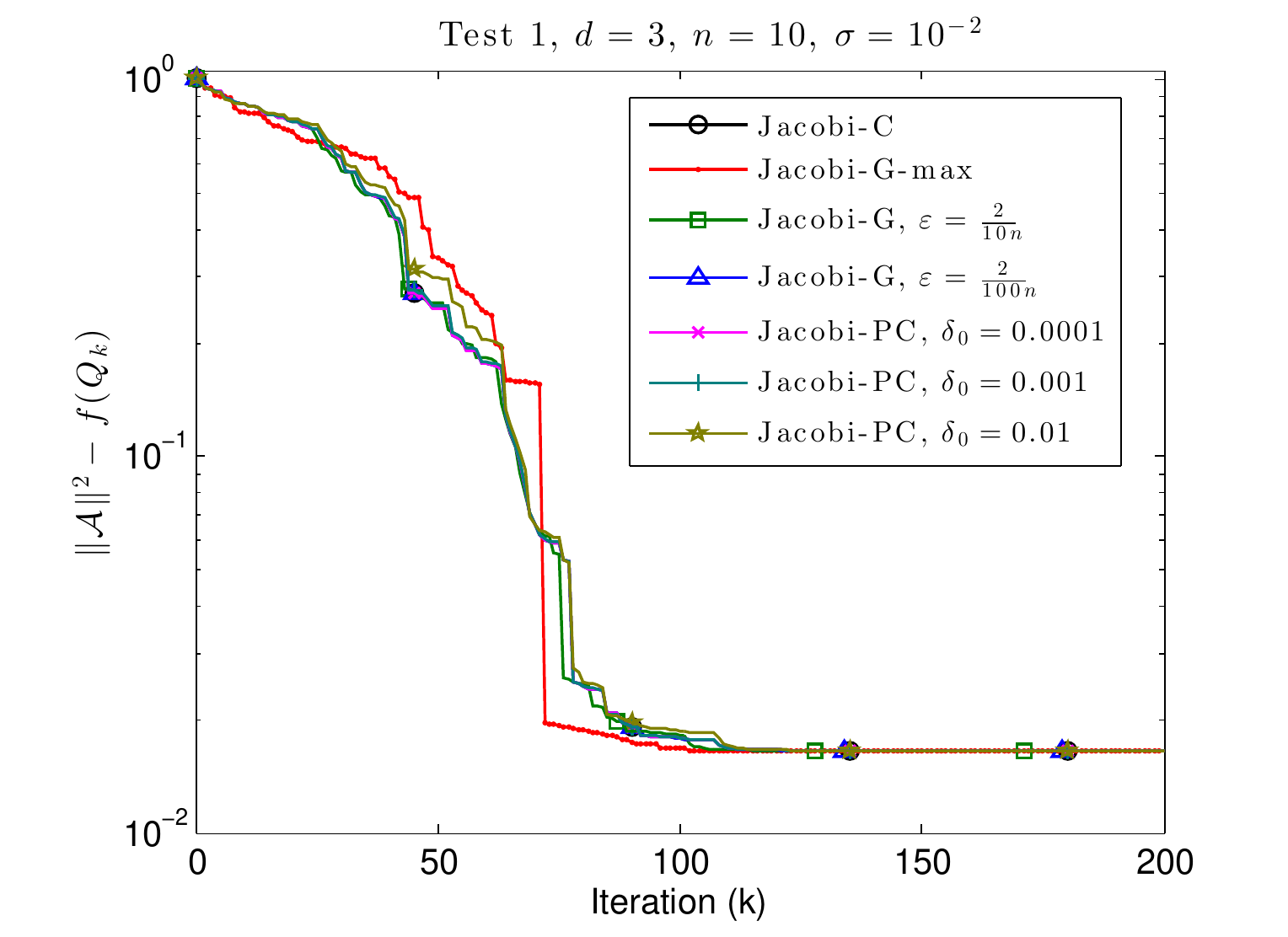}}\!\!\!
\subfloat[$4$-th order]{\includegraphics[width=0.5\textwidth]{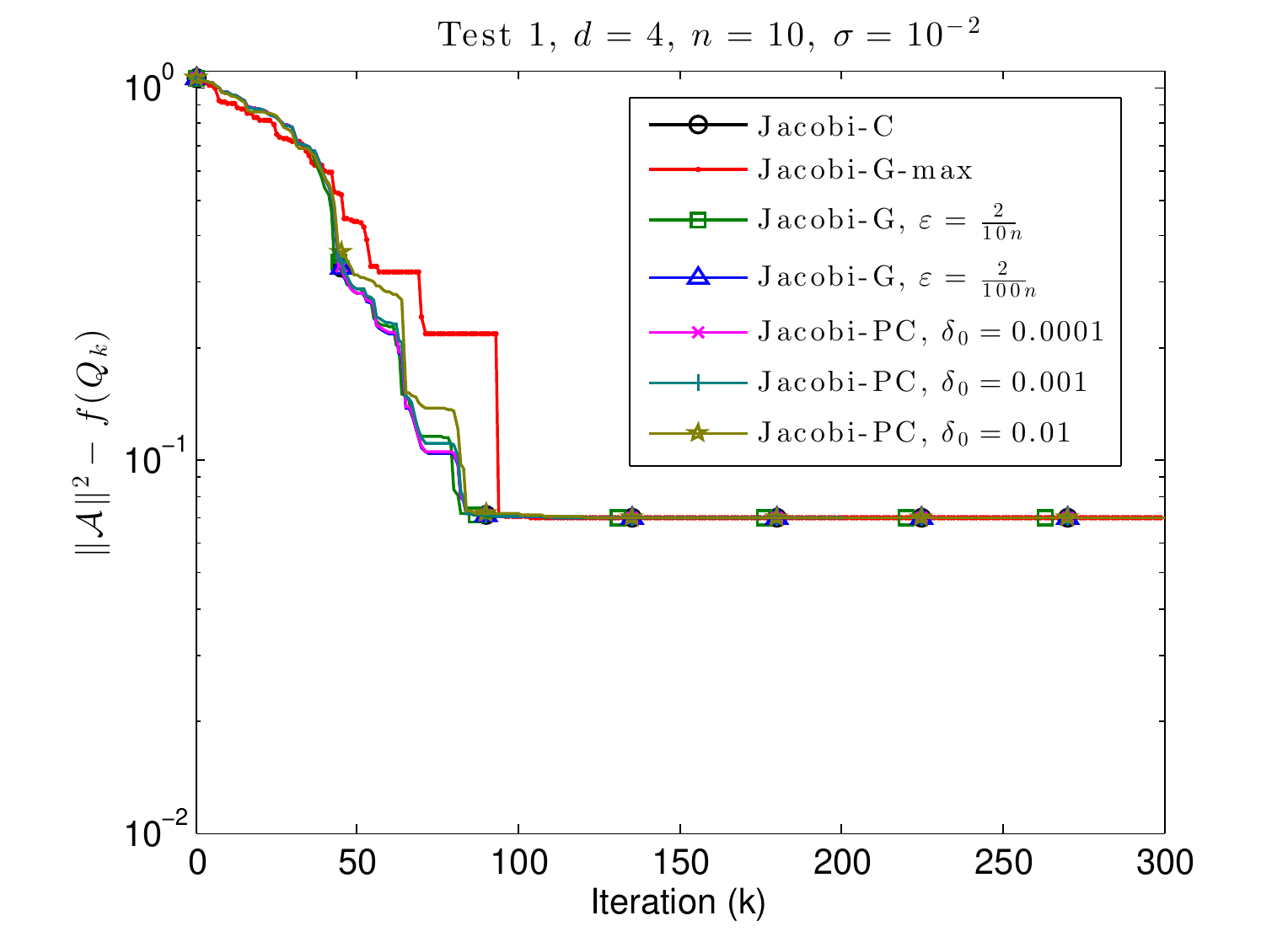}}
\caption{Equal values on the diagonal, higher noise.}
\label{fig:test1_sig2}
\end{figure}

As we see in  \cref{fig:test1_sig4,fig:test1_sig2}, in all the examples all the methods converge to the same cost function value.
We observe that the behavior of the Jacobi-PC algorithm it not too different from the behavior of the Jacobi-C algorithm.

The convergence of Jacobi-G-max is the fastest, but the difference is marginal.
Also, the Jacobi-G-max is typically slower in the beginning, but accelerates when the algorithm is closer to the local maximum.
Finally, if $\varepsilon$ is small, the behavior of Jacobi-G is almost indistinguishable from Jacobi-C, as pointed out in \cref{prop:jacobi_threhold}.

\subsection{Test 2: different values on the diagonal}\label{sec:test-2}
In this subsection, we consider $10\times 10\times 10$ and $10\times 10 \times 10\times 10$ tensors where the diagonal values are given by
\[
\tenselem{D}_{i\ldots i} = \frac{i}{\sqrt{385}}.
\]
We plot the results in \cref{fig:test2_sig4,fig:test2_sig2}.
\begin{figure}[tbhp]
\centering
\subfloat[$3$-rd order]{\includegraphics[width=0.5\textwidth]{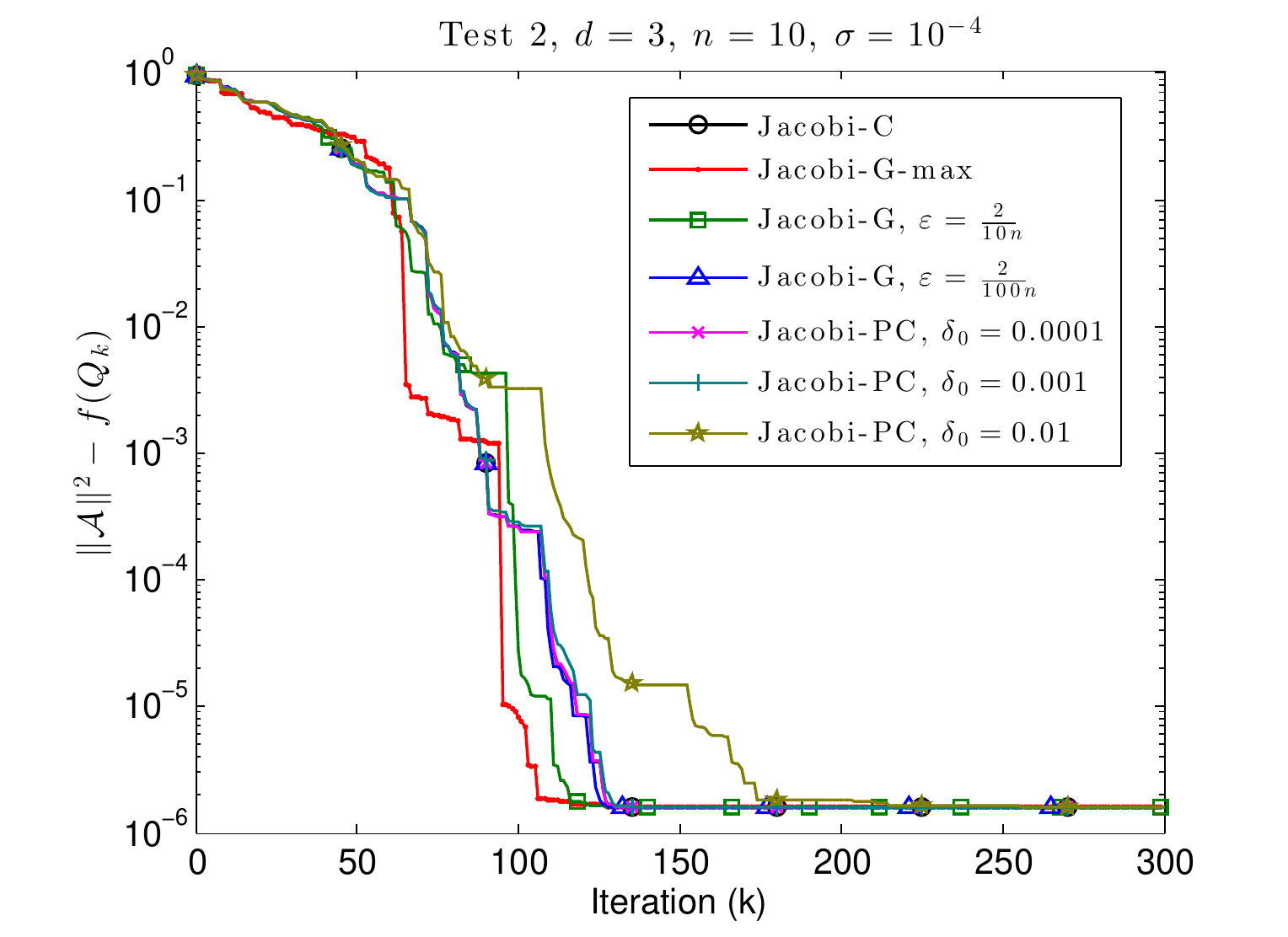}}\!\!\!
\subfloat[$4$-th order]{\includegraphics[width=0.5\textwidth]{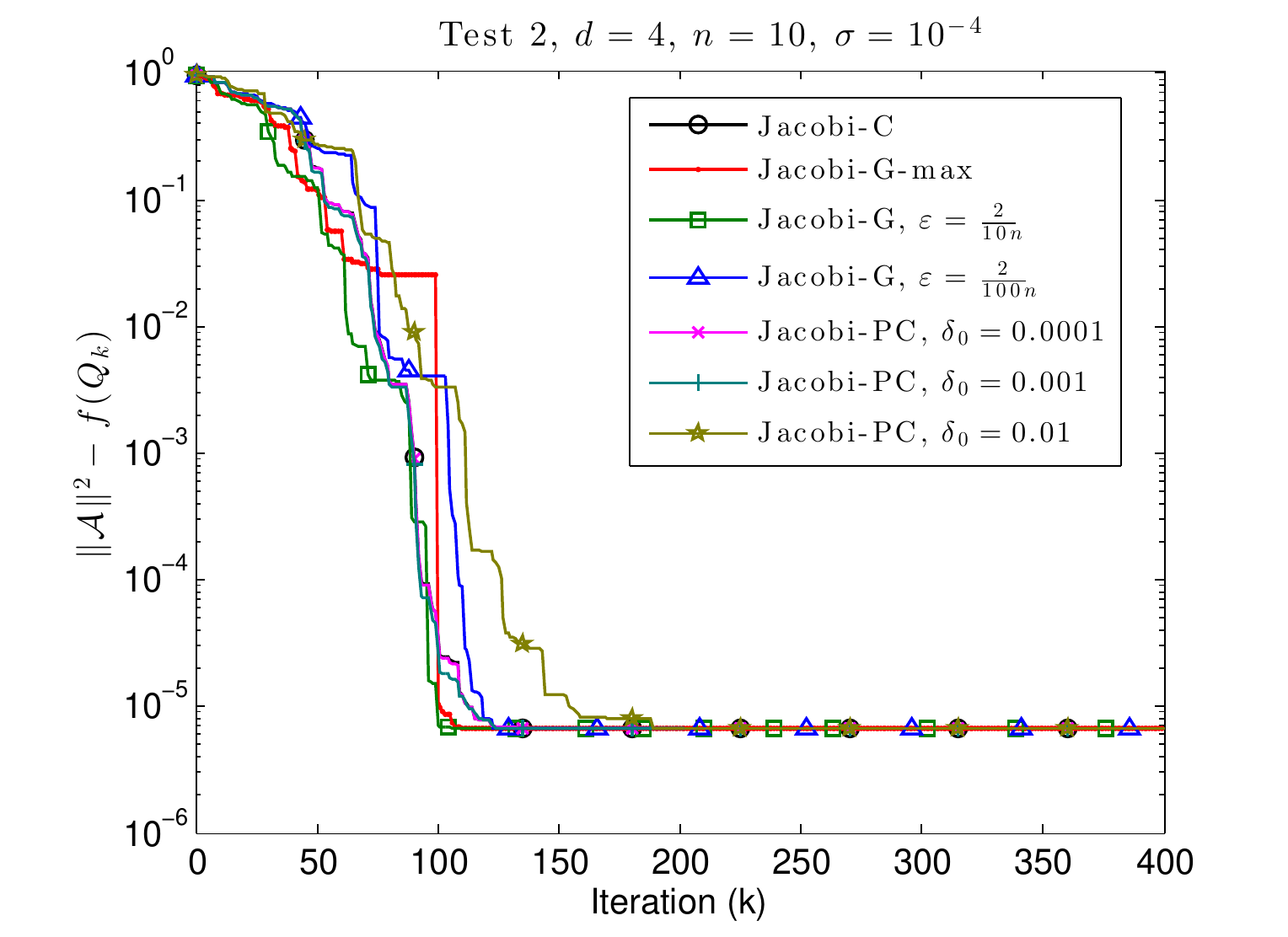}}
\caption{Different values on the diagonal, small noise.}
\label{fig:test2_sig4}
\end{figure}
\begin{figure}[tbhp]
\centering
\subfloat[$3$-rd order]{\includegraphics[width=0.5\textwidth]{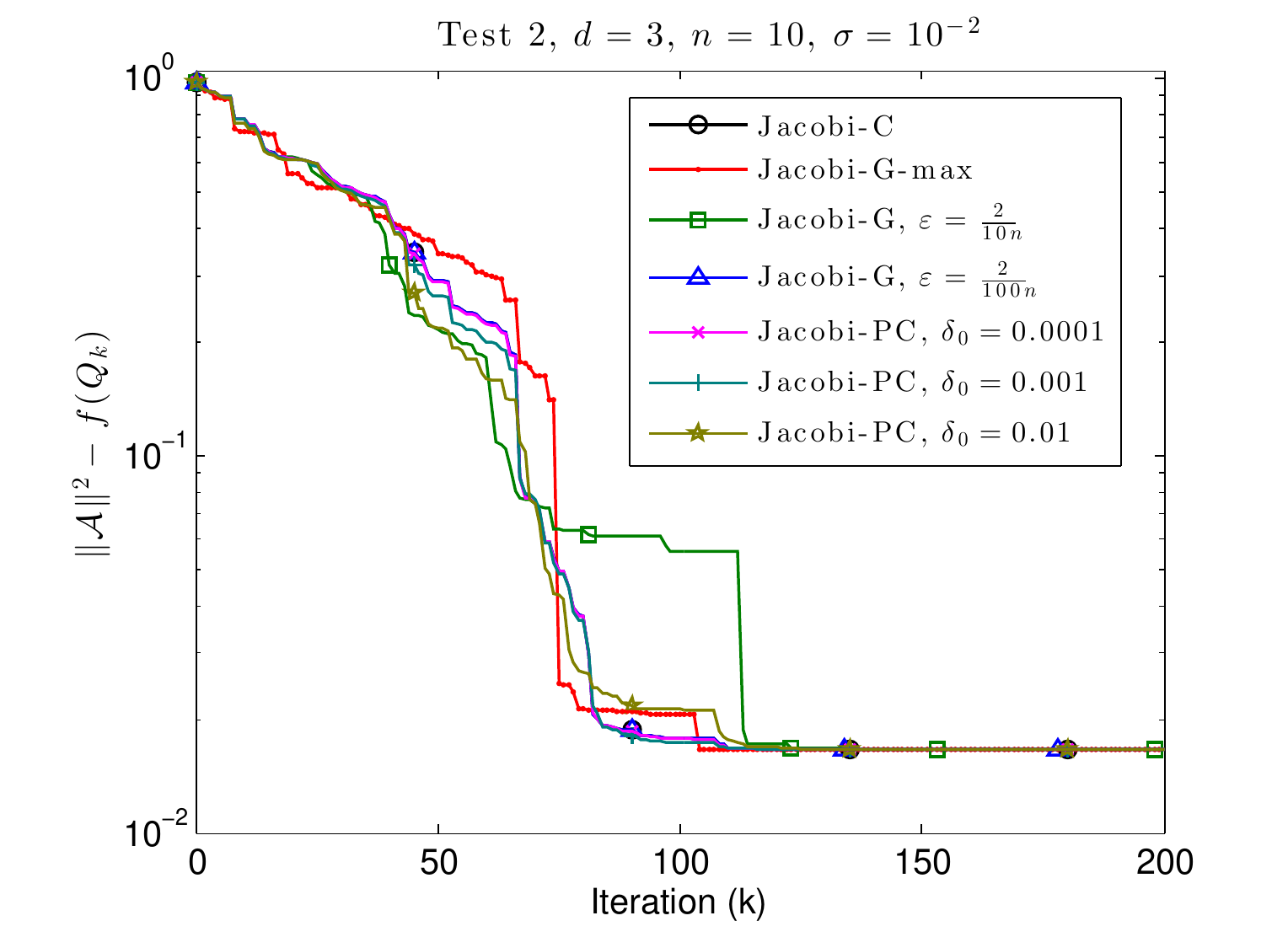}}\!\!\!
\subfloat[$4$-th order]{\includegraphics[width=0.5\textwidth]{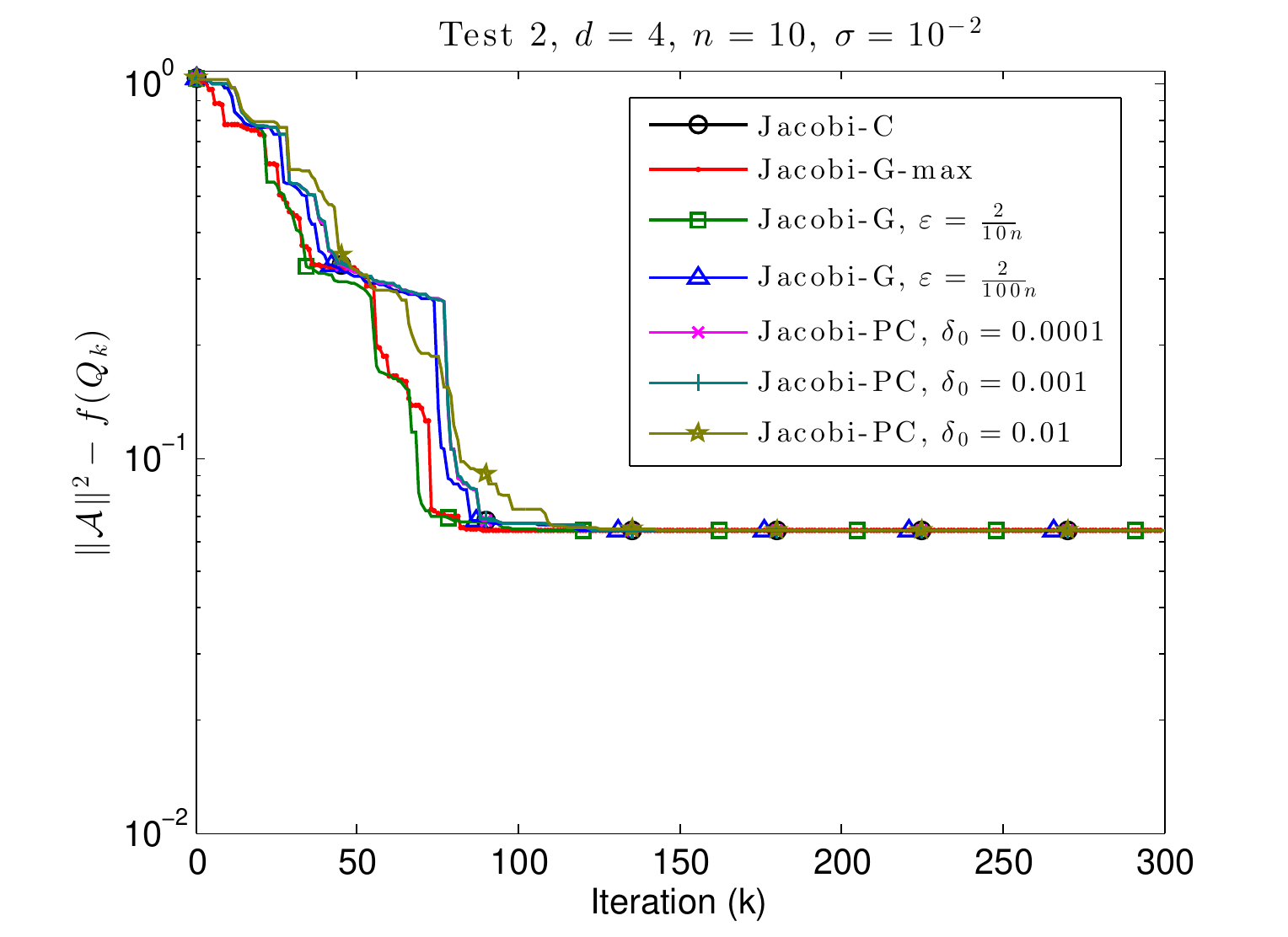}}
\caption{Different values on the diagonal, higher noise.}
\label{fig:test2_sig2}
\end{figure}

In \cref{fig:test2_sig4,fig:test2_sig2} we see that this scenario is less favorable for Jacobi-PC: if the value of $\delta_0$ is too high, then it slows down the convergence of the algorithm.
We also see that typically the Jacobi-G algorithms are the fastest, but the difference with Jacobi-C is not significant again.
Also for small values of $\varepsilon$, the behavior of Jacobi-G resembles the behavior of Jacobi-C.

\subsection{High noise and local minima}
\ku In this subsection, we consider the case of  \fin high noise.
\jl We repeat only the $4$th-order experiments (for a single tensor) from  \cref{sec:test-2} except with $\sigma = 10^{-1}$. We take two different realizations of $\tens{E}$ and plot the results in \cref{fig:test2_sig1}. \fin
\begin{figure}[tbhp]
\centering
\subfloat[$1$-st realization]{\includegraphics[width=0.5\textwidth]{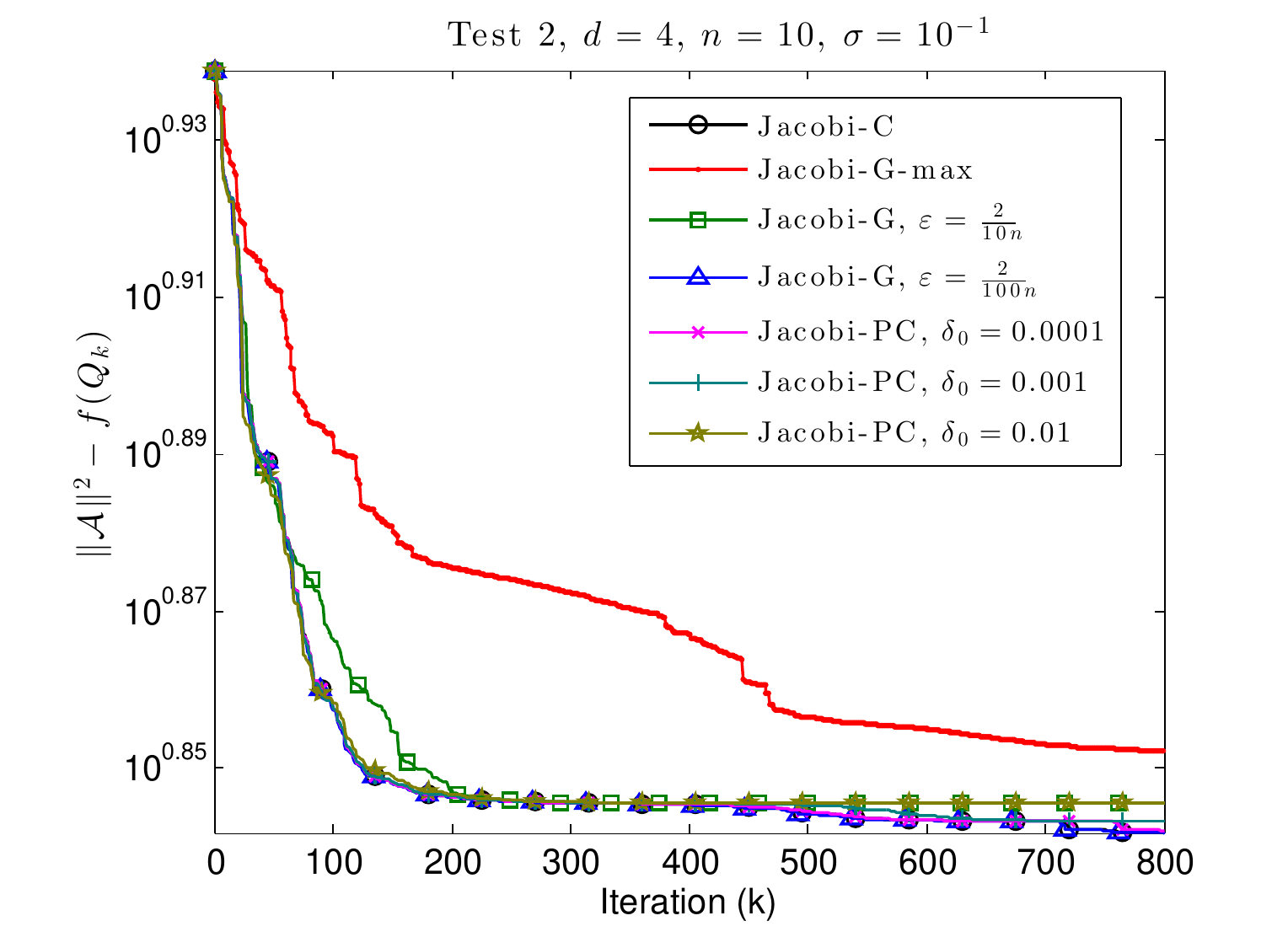}}\!\!\!
\subfloat[$2$-nd realization]{\includegraphics[width=0.5\textwidth]{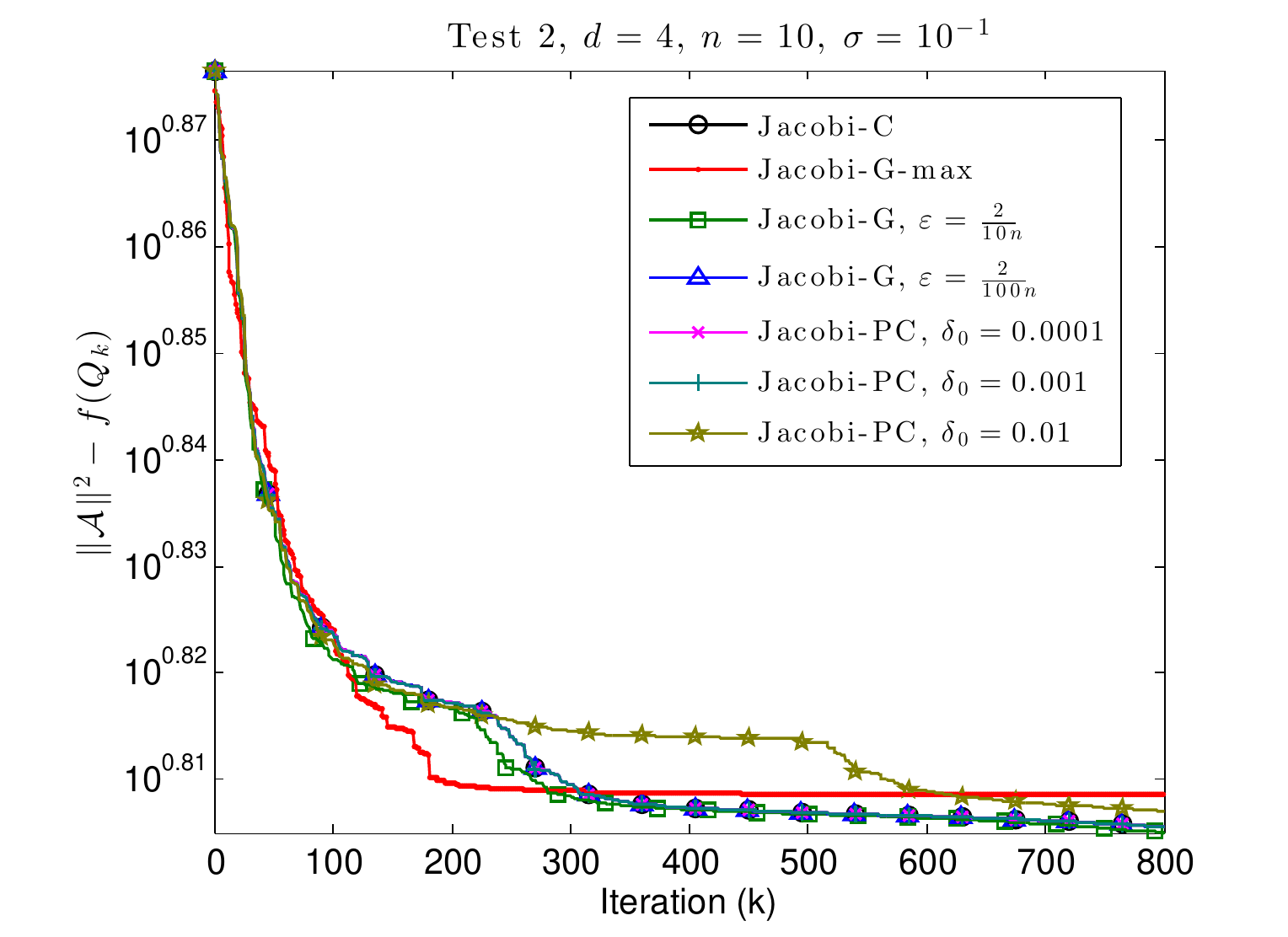}}
\caption{Different values on the diagonal, high noise, different noise realizations.}
\label{fig:test2_sig1}
\end{figure}

In \cref{fig:test2_sig1}, we see that the behavior of the algorithms is more erratic, and they may converge to different cost function values.
This is explained by the non-convexity of the problem and presence of different local minima, which is typical for the tensor approximation problems \cite{Isht11:local}.
Next, the Jacobi-G-max algorithm here has the worst performance.
This is also explained well by the non-convexity of the problem, because the compatibility of the Jacobi rotation with the gradient
(eqn. \eqref{eq:pair_selection_gradient}) may not be optimal.

\ku
\subsection{Simultaneous diagonalization}
We conclude the section by a small example of simultaneous diagonalization.
We take 4th order $10\times 10\times 10\times 10$ tensor $\tens{A}$ generated as in \crefrange{sec:test-1}{sec:test-2} (for the noise level $\sigma = 10^{-2}$), and consider its $10$ slices $\tens{B}^{(1)}, \ldots,\tens{B}^{(10)} \in \RR^{n\times n \times n}$ along the last dimension, \emph{i.e.}
\[
\tenselem{B}^{(i)}_{k,l,s} = \tenselem{A}_{k,l,s,i}.
\]
Then, we perform the joint diagonalization of tensors $\tens{B}^{(1)}, \ldots,\tens{B}^{(m)}$  (for $m=10$) and run the same algorithms as in the previous experiments, but for the cost function in the case of simultaneous diagonalization. The results are plotted in \cref{fig:test12_simt_sig2}.
\begin{figure}[tbhp]
\centering
\subfloat[Test $1$]{\includegraphics[width=0.5\textwidth]{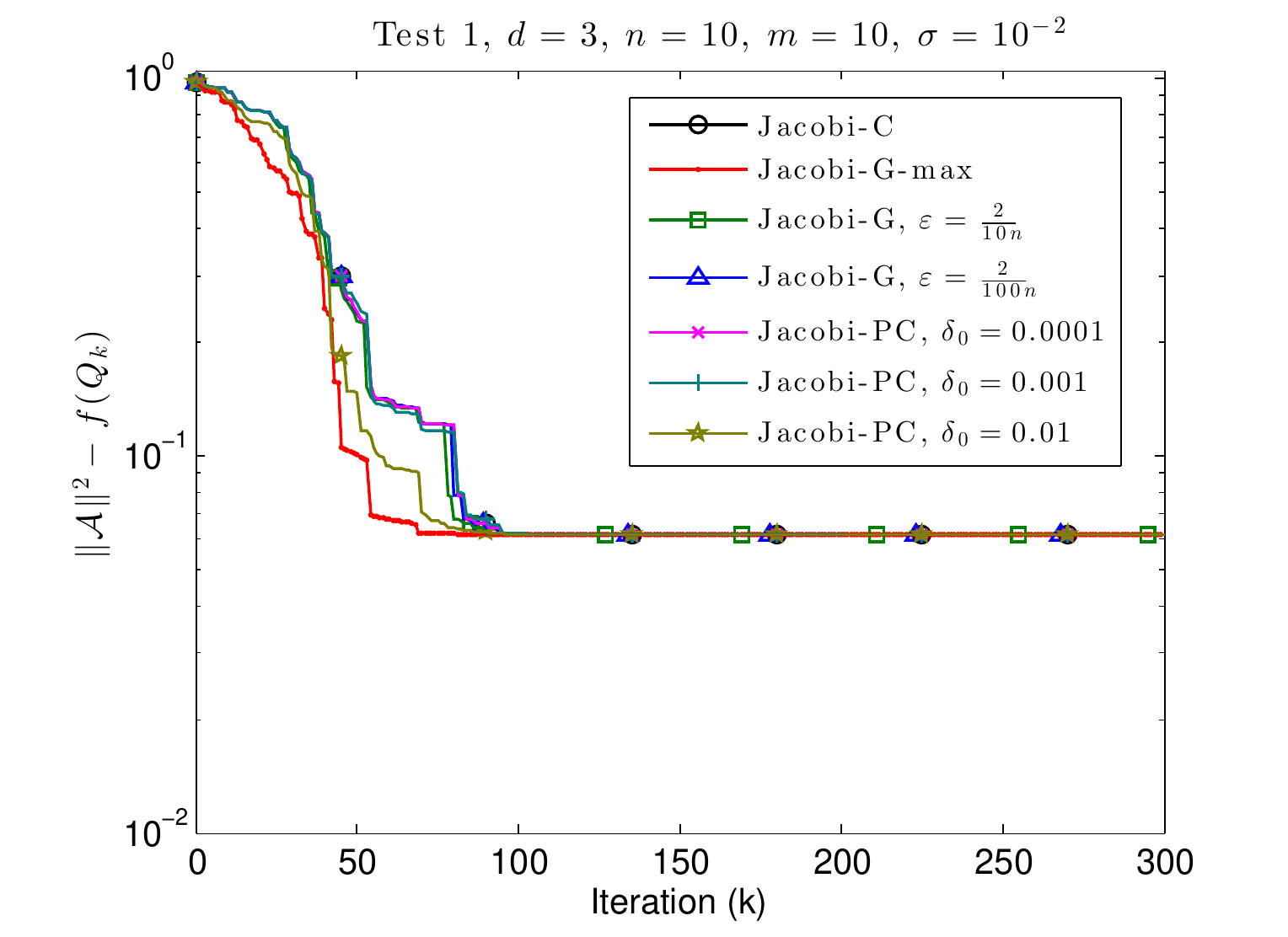}}\!\!\!
\subfloat[Test $2$]{\includegraphics[width=0.5\textwidth]{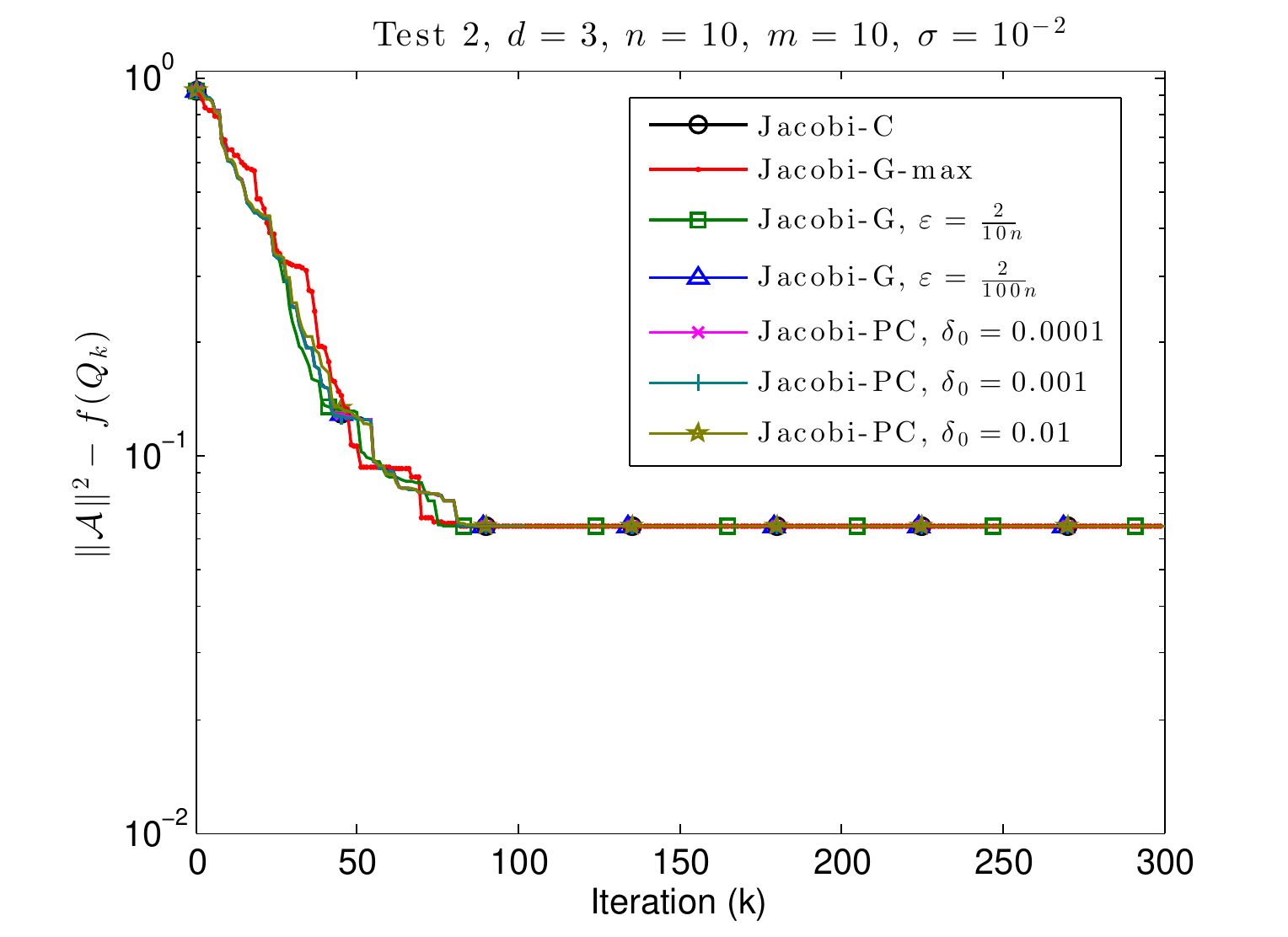}}
\caption{Simultaneous diagonlization of tensor slices.}
\label{fig:test12_simt_sig2}
\end{figure}

The results in \cref{fig:test12_simt_sig2} exhibit a similar behavior to the results in \crefrange{sec:test-1}{sec:test-2}.
When comparing the results of single tensor diagonalization for the same tensors, (see \cref{fig:test1_sig2} and \cref{fig:test2_sig2}, subfigures (b)), we can see that the results are comparable, and even the simultaneous diagonalization may yield a slightly higher cost function value. But the cost function  in this case is different because the tensor is not rotated along the last mode.
\fin

\section{Conclusions}
\label{sec:conclusions}
We showed that by modifying the well-known Jacobi CoM algorithm \cite{Como92:elsevier,Como94:ifac}
for {orthogonal symmetric tensor diagonalization problem},
it is possible to prove its global convergence.
{The global convergence of Jacobi-G algorithm \cite{IshtAV13:simax} 
is proved for the case of simultaneous orthogonal symmetric matrix (or 3rd-order tensor) diagonalization.
The global convergence for 4th-order case is still unknown.}
Our new proximal-type algorithm Jacobi-PC is globally convergent for a wide range of optimization problems, and shows a good performance in the numerical experiments.

%
%Some conclusions here.
%
%\section*{Acknowledgments}

\bibliographystyle{siamplain}
\bibliography{jacobi_simax}
%{../jacobi}
\end{document}